\documentclass{article}

\usepackage{amsmath,amsfonts,bm}

\def\eqref#1{equation~\ref{#1}}

\def\ceil#1{\lceil #1 \rceil}

\def\1{\bm{1}}

\DeclareMathAlphabet{\mathsfit}{\encodingdefault}{\sfdefault}{m}{sl}
\SetMathAlphabet{\mathsfit}{bold}{\encodingdefault}{\sfdefault}{bx}{n}

\newcommand{\R}{\mathbb{R}}

\newcommand{\KL}{D_{\mathrm{KL}}}

\DeclareMathOperator*{\argmax}{arg\,max}
\DeclareMathOperator*{\argmin}{arg\,min}

\usepackage{natbib}

\usepackage{hyperref}

\usepackage{microtype}
\usepackage{graphicx}
\usepackage{subfigure}
\usepackage{booktabs} 
\usepackage{algpseudocode}
\usepackage{subcaption}

\usepackage{tcolorbox}
\newtcolorbox{thmbox}{colback=cyan!5,colframe=white,top=2pt,bottom=2pt,left=0pt,right=2pt}
\newtcolorbox{questionbox}{colback=red!5!white,colframe=white}
\renewcommand{\KL}[2]{\mathrm{KL}\left(#1 || #2\right)}
\newcommand{\LSE}[1]{\mathrm{LSE}\left(#1\right)}

\usepackage{amsmath}
\usepackage{amssymb}
\usepackage{mathtools}
\usepackage{amsthm}

\usepackage[capitalize,noabbrev]{cleveref}

\theoremstyle{plain}
\newtheorem{theorem}{Theorem}
\newtheorem{proposition}{Proposition}

\theoremstyle{definition}
\newtheorem{definition}{Definition}

\theoremstyle{remark}
\newtheorem{remark}{Remark}
\allowdisplaybreaks

\usepackage[textsize=tiny]{todonotes}

\usepackage[arxiv]{icml2023}

\icmltitlerunning{A Sinkhorn-type algorithm for inequality-constrained optimal transport}

\usepackage[textsize=tiny]{todonotes}

\begin{document}

\twocolumn[
\icmltitle{A Sinkhorn-type Algorithm for Constrained Optimal Transport
}

\begin{icmlauthorlist}
\icmlauthor{Xun Tang}{yyy}
\icmlauthor{Holakou Rahmanian}{comp}
\icmlauthor{Michael Shavlovsky}{comp}
\icmlauthor{Kiran Koshy Thekumparampil}{comp}
\icmlauthor{Tesi Xiao}{comp}
\icmlauthor{Lexing Ying}{yyy}

\icmlaffiliation{yyy}{Stanford Univerisity}
\icmlaffiliation{comp}{Amazon}

\icmlcorrespondingauthor{Xun Tang}{xuntang@stanford.edu}
\icmlcorrespondingauthor{Lexing Ying}{lexing@stanford.edu}

\end{icmlauthorlist}

\icmlkeywords{Optimal Transport, Sinkhorn Algorithm, Applied Mathematics, Information Retrieval}
\vskip 0.3in

]

\printAffiliationsAndNotice{}

\begin{abstract}
    Entropic optimal transport (OT) and the Sinkhorn algorithm have made it practical for machine learning practitioners to perform the fundamental task of calculating transport distance between statistical distributions. In this work, we focus on a general class of OT problems under a combination of equality and inequality constraints. We derive the corresponding entropy regularization formulation and introduce a Sinkhorn-type algorithm for such constrained OT problems supported by theoretical guarantees. We first bound the approximation error when solving the problem through entropic regularization, which reduces exponentially with the increase of the regularization parameter. Furthermore, we prove a sublinear first-order convergence rate of the proposed Sinkhorn-type algorithm in the dual space by characterizing the optimization procedure with a Lyapunov function. To achieve fast and higher-order convergence under weak entropy regularization, we augment the Sinkhorn-type algorithm with dynamic regularization scheduling and second-order acceleration. Overall, this work systematically combines recent theoretical and numerical advances in entropic optimal transport with the constrained case, allowing practitioners to derive approximate transport plans in complex scenarios.
\end{abstract}

\section{Introduction}
\label{sec: intro}
Obtaining the optimal transport (OT) \citep{villani2009optimal, linial1998deterministic,peyre2017computational} plan between statistical distributions is an important subroutine in machine learning \citep{sandler2011nonnegative,jitkrittum2016interpretable,arjovsky2017wasserstein,salimans2018improving,genevay2018learning,chen2020graph,fatras2021unbalanced}. In this work, we focus on an optimal transportation problem with a combination of inequality and equality constraints. A typical example is an OT problem with one inequality constraint:
\begin{equation}\label{eqn: constrained OT example}
    \min_{
\substack{P: P\1=r, P^{\top}\1 = c, P \geq 0}}
  C\cdot P, \, \text{such that \(D \cdot P \geq t\)},
\end{equation}
where \(\cdot\) stand for entry-wise inner product, \(C \in \mathbb{R}^{n \times n}\) is the cost matrix, \(t \in \R, D \in \mathbb{R}^{n \times n}\) encode the inequality constraint, and \(c,r \in \mathbb{R}^{n}\) are respectively the source and target density. As an illustration, in Figure \ref{fig:1D_transport_evolution}, we plot the optimal transport plan between 1D distributions when the main cost is the transport cost induced by the \(l_1\) Manhattan distance, but an inequality constraint is placed on the transport cost induced by Euclidean distance. One can see that solving inequality-constrained OT problems in \eqref{eqn: constrained OT example} allows one to obtain transport maps under more complex geometric structures.

For the unconstrained OT problem, the most important recent breakthrough is the introduction of entropic regularization and the resultant Sinkhorn algorithm \citep{yule1912methods, sinkhorn1964relationship, cuturi2013sinkhorn}. With simple matrix scaling steps, the Sinkhorn algorithm gives an approximate OT solution in near-linear time \citep{altschuler2017near}, which fuels the wide adoption in the machine learning community. This prompts the following natural question:
\begin{questionbox}
\begin{center}
    Is there an extension of the Sinkhorn algorithm to optimal transport problems with constraints?
\end{center}
\end{questionbox}

This paper answers the above question in the affirmative. The key insight underlying the development in this work is that the entropic optimal transport problem is an instance of entropic linear programming \citep{fang1992unconstrained}, from which it is natural to extend entropy regularization to the constrained case. For an illustration, the example in \eqref{eqn: constrained OT example} has the following relaxation:
\begin{thmbox}
\begin{equation}\label{eqn: entropic constrained OT example}
\begin{aligned}
    &\min_{
\substack{P, s: P\1=r, P^{\top}\1 = c, P \geq 0, s \geq 0}}
  C\cdot P + \frac{1}{\eta}H(P, s),\\
  &\text{subject to \(D \cdot P - s = t\),}
\end{aligned}
\end{equation}
\end{thmbox}
where \(s\) is the slack variable, \(\eta > 0\) is the entropy regularization parameter, and \(H(P,s) := \sum_{ij}p_{ij}\log{p_{ij}} + s\log(s)\) is the entropy regularization term. The primal-dual form of \eqref{eqn: entropic constrained OT example} leads to a direct generalization of the Sinkhorn algorithm to the constrained case, which we develop in Section \ref{sec: variational} and Section \ref{sec: alg}. 

\begin{figure*}
    \centering
    \includegraphics[width = 0.99\textwidth]{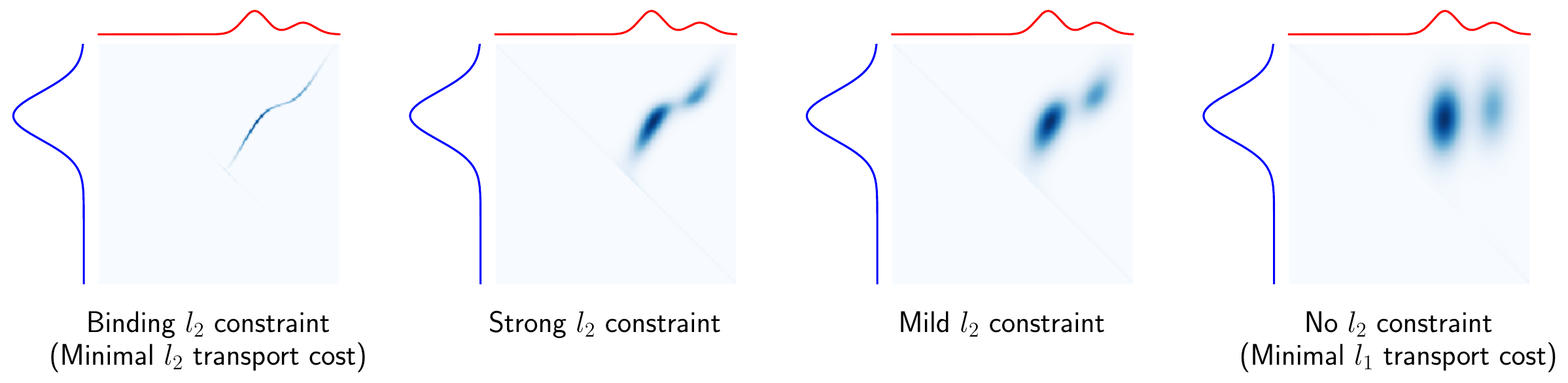}
    \caption{Illustration of 1D optimal transport under different inequality constraints. By incrementally decreasing the upper bound for the Euclidean distance transport cost, the transport plan evolves from minimizing the Euclidean distance transport cost to minimizing the Manhattan distance (\(l_1\)) transport cost. The transport plan is computed with the Sinkhorn-type procedure in Algorithm \ref{alg:1}.}
    \label{fig:1D_transport_evolution}
\end{figure*}
\subsection{Contributions}
In addition to the introduction of a novel Sinkhorn-type algorithm for constrained OT, this paper systematically generalizes existing theoretical results as well as practical acceleration techniques from the existing entropic OT literature. In particular,
\begin{itemize}
    \item As a corollary to the analysis in \citet{weed2018explicit}, we show that entropic optimal transport in the constrained case is exponentially close to the optimal solution.
    \item By extending the result in \citet{altschuler2017near}, we show that a Sinkhorn-type algorithm reaches approximate first-order stationarity in polynomial time.
    \item We further improve practical performance by introducing dynamic regularization scheduling in \citet{chen2023exponential} to the Sinkhorn-type algorithm.
    \item We extend the approximate sparsity analysis in \citet{tang2024accelerating} to the constrained OT case and introduce an accelerated second-order algorithm through sparse Newton iterations. 
    \item We introduce a variational formulation of the task, which naturally enables the use of first-order methods such as in \citet{dvurechensky2018computational}.
\end{itemize}
\subsection{Related literature} 

\paragraph{Constrained optimal transport} The numerical treatment for constrained OT focuses on special cases such as capacity constraints \citep{korman2013insights, korman2015optimal}, multi-marginal transport \citep{gangbo1998optimal,buttazzo2012optimal,benamou2015iterative,pass2015multi,khoo2020semidefinite}, and martingale optimal transport \citep{tan2013optimal,beiglbock2013model,galichon2014stochastic, dolinsky2014martingale, guo2019computational}. In terms of numerical methods for constrained OT, \citet{benamou2015iterative} defines an iterative Bregman projection approach for equality constraints and a Bregman–Dykstra iteration for general constraints. In Appendix \ref{appendix: bregman}, we show that the iterative Bregman projection is equivalent to our proposed update step in the equality case. Our contribution in relation to \citet{benamou2015iterative} is threefold. First, for equality constraints, we propose a generalized Sinkhorn-type algorithm that implements the iterative Bregman projection, whereas previous work only defines the projection step in special cases. Second, for inequality constraint, our approach is novel and enjoys theoretical guarantee as seen in Theorem \ref{thm: exp convergence OT} and Theorem \ref{thm: OT linear convergence}. Finally, the Bregman–Dykstra iteration in \citep{benamou2015iterative} is only implemented when a simple closed-form solution exists for each projection step, whereas our proposed method applies to general cases.

\paragraph{OT in machine learning} 
A substantial amount of research literature exists on the use of optimal transport in different areas of machine learning. Notable applications include statistical learning \citep{kolouri2017optimal,vayer2018optimal,genevay2019sample,luise2018differential,oneto2020exploiting,huynh2020otlda}, domain adaptation \citep{Fernando_2013_ICCV,redko2017theoretical,courty2017joint,alvarez2018structured,nguyen2022improving,nguyen2021most,xu2022unsupervised,turrisi2022multi}, and using optimal transport distance in designing training targets \citep{genevay2017gan,bousquet2017optimal,sanjabi2018convergence,deshpande2019max,lei2019geometric,patrini2020sinkhorn,onken2021ot}. The usage of entropic optimal transport in the constrained case allows practitioners to query transport plans which satisfy more complex structure than in the unconstrained case, which are beneficial across all the applications above. 

\paragraph{Acceleration for Sinkhorn}
There is a considerable body of work in speeding up the runtime of the Sinkhorn algorithm, and a significant portion of our numerical treatment is devoted to the extension of these algorithms to the constrained case. In addition to \citep{chen2023exponential,tang2024accelerating} to be discussed in Section \ref{sec: alg}, we mention a few noteworthy methods for acceleration.
Randomized or greedy row/column scaling \citep{genevay2016stochastic, altschuler2017near} can be directly extended to the constrained case by the inclusion of the dual variables for the linear constraints. Likewise, methods based on the variational form of entropic OT, such as Nesterov acceleration \citep{dvurechensky2018computational} and mirror descent \citep{lin2019efficient,kemertas2023efficient}, can be used in the constrained case by considering the Lyapunov function introduced in Section \ref{sec: variational}. While the majority of existing techniques can be directly extended to the constrained case, one notable exception is the class of methods based on approximation to the kernel \(K = \exp(-C\eta)\) \citep{deriche1993recursively,solomon2015convolutional,bonneel2016wasserstein, altschuler2019massively,scetbon2020linear,lakshmanan2022nonequispaced,huguet2023geodesic,li2023importance}. The main obstruction from performing kernel approximation in the constrained case is that the kernel is changing throughout the optimization process due to update in the constraint associated dual variables. As the kernel compression step is costly and often done in the offline stage, it poses a significant challenge if one has to update the kernel approximation with each dual variable update dynamically.

\section{Theoretical foundations of constrained entropic optimal transport}\label{sec: variational}
This section summarizes the entropic optimal transport in the constrained case.

\paragraph{Notations} We use \(M \cdot M' := \sum_{ij}m_{ij}m'_{ij}\) to denote the entry-wise inner product. For a matrix \(M\), the notation \(\log{M}\) stands for entry-wise logarithm, and similarly \(\exp(M)\) denotes entry-wise exponential. We use the symbol \(\lVert M\rVert_{1}\) to denote the entry-wise \(l_1\) norm, i.e. \(\lVert M\rVert_{1} := \lVert \mathrm{vec}(M)\rVert = \sum_{ij} \lvert m_{ij} \rvert \). The \(\lVert M \rVert_{\infty}\) and \(\lVert M \rVert_{2}\) norms are defined likewise as the entry-wise \(l_\infty\) and \(l_2\) norms, respectively. The notation \(\mathbf{1}_{n\times n}\) is the all-one \(n\times n\) matrix, and the notation \(\mathbf{1}\) denotes the all-one vector of appropriate size.

\paragraph{Background}
For simplicity, we assume the target and source density \(r, c \in \R^n\) satisfies \(\sum_i r_i = \sum_j c_j = 1\).
For any inequality constraint of the form \(D \cdot P \geq t\), note one can convert the condition to \((D - t \mathbf{1}_{n \times n}) \cdot P \geq 0\), and similarly in the equality case. Let \(K\) and \(L\) respectively denote the number of inequality and equality constraints. By the construction above, for each optimal transport problem under linear constraint, there exists \(D_{1}, \ldots D_{K}, D_{K+1}, \ldots, D_{K+L} \in \R^{n \times n}\) so that the linear constraints are encoded by \(\mathcal{I}, \mathcal{E}\) where
\begin{equation}\label{eqn: def of constraint}
    \begin{aligned}
        \mathcal{I} := &\bigcap_{k = 1, \ldots, K} \{P \mid D_{k} \cdot P \geq 0\},
        \\
        \mathcal{E} := &\bigcap_{l = 1, \ldots, L} \{P \mid D_{l+K} \cdot P = 0\}.
    \end{aligned}
\end{equation}

We summarize the \emph{general form of constrained optimal transport} by the following the following linear program (LP):
\begin{equation}\label{eqn: constrained OT general form}
    \min_{
\substack{P: P\1=r, P^{\top}\1 = c, P \geq 0}}
  C\cdot P, \, \text{such that \(P \in \mathcal{S}\)},
\end{equation}
where \(\mathcal{S} := \mathcal{I} \cap \mathcal{E}\) for \(\mathcal{I}, \mathcal{E}\) defined in \eqref{eqn: def of constraint}.

\paragraph{Entropic optimal transport under constraint}
Under entropic linear programming \citep{fang1992unconstrained}, one can write down the formulation for entropic optimal transport under general linear constraints. For \(k = 1, \ldots, K\), we define \(s_{k}\) to be the slack variable corresponding to \(D_{k} \cdot P\). The 
\emph{constrained entropic optimal transport} follows from the following equation:
\begin{equation}\label{eqn: entropic constrained OT general form}
\begin{aligned}
    &\min_{
\substack{P, s: P\1=r,\, P^{\top}\1 = c,\,\\ P \geq 0,\, s_{k} \geq 0 \, \forall\, k}} &C\cdot P + \frac{1}{\eta}H(P, s_1, \ldots, s_K),\\
  &\text{subject to} &D_{k} \cdot P = s_{k} \text{ for \(k = 1, \ldots, K\).}\\
  &  &D_{l+K} \cdot P = 0 \text{ for \(l = 1, \ldots, L\),}
\end{aligned}
\end{equation}
where the entropy term is defined by \(H(P, s_1, \ldots, s_K) = \sum_{ij}p_{ij} \log(p_{ij}) + \sum_{k = 1}^{K} s_{k} \log s_{k}\).

We motivate the optimization task in \eqref{eqn: entropic constrained OT general form} by Theorem \ref{thm: exp convergence OT}, which shows that the entropy-regularized optimal solution is exponentially close to the optimal solution: 

\begin{theorem}\label{thm: exp convergence OT}
    For simplicity, assume that \(\sum_{i}r_{i} = \sum_{j}c_j = 1\), the LP in \eqref{eqn: constrained OT general form} has a unique solution \(P^{\star}\), and assume that \(\lVert D_{k} \rVert_{\infty} \leq 1\) for \(k = 1,\ldots, K\). Denote \(P^{\star}_{\eta}\) as the unique solution to \eqref{eqn: entropic constrained OT general form}. There exists a constant \(\Delta\), depending only on the LP in \eqref{eqn: constrained OT general form}, so that the following holds for \(\eta \geq \frac{(K+1)(1 + \ln (4 n^2 (K+1)))}{\Delta}\):
    \[
    \lVert P^{\star}_{\eta} - P^{\star} \rVert_1 \leq 8n^{\frac2{K+1}}(K+1)\exp\left(-\eta\frac{\Delta}{K+1}\right).
    \]
\end{theorem}
The definition of \(\Delta\) and the proof are deferred to Appendix \ref{appendix: proof of OT exp convergence}. In this work, we only consider examples for which \(K = O(1)\), in which case the bound in Theorem \ref{thm: exp convergence OT} does not significantly differ from the unconstrained case. It should be noted that the dependency on $n$ goes down with an increasing number of inequality constraints $K$, but one can check the bound is monotonically worse as $K$ increases. Interpreting the bound could mean that adding more inequality constraints limits the degree of freedom but may also potentially amplify the error caused by entropy regularization.

\paragraph{Variational formulation under entropic regularization}
By introducing the Lagrangian variable and using the minimax theorem (for a detailed derivation, see Appendix \ref{sec: primal versus primal dual}), we formulate the associated primal-dual problem to \eqref{eqn: entropic constrained OT general form} as:
\begin{equation*}
\begin{aligned}
    &\max_{x,y,a} \min_{P,s} \big[ L(P, s, x, y, a) \\ :=&\frac{1}{\eta} P \cdot \log P + C\cdot P - x\cdot(P\1-r) - y\cdot(P^{\top}\1-c)\\
    + &\frac{1}{\eta}\sum_{k = 1}^{K}s_{k} \log{s_{k}} + \sum_{k=1}^{K}a_{k}s_{k} -\sum_{m = 1}^{K+L}a_{m}(D_{m} \cdot P) \big],
\end{aligned}
\end{equation*}
where \(s = (s_{1}, \ldots, s_{K})\) is the shorthand for slack variable, \(a = (a_{1}, \ldots, a_{K}, a_{K+1}, \ldots, a_{K+L})\) is the shorthand for the constraint dual variables (excluding the original row/column constraints).
 
By eliminating \(P, s\) (see Appendix \ref{sec: primal versus primal dual}), the function \(f(x,y,a)  := \min_{P,s} L(P,s, x, y, a)\) admits the following form
\begin{equation}\label{eqn: dual form}
    \begin{aligned}
        &f(x,y,a)\\
        =&\frac{-1}{\eta}\sum_{ij}\exp{\Big(\eta(-C_{ij} + \sum_{m = 1}^{L + K}a_{m}\left(D_{m}\right)_{ij}+ x_i + y_{j}) - 1 \Big)}\\
        + &\sum_{i}x_{i}r_{i} + \sum_{j}y_{j}c_{j} - \frac{1}{\eta}\sum_{k = 1}^{K} {\exp(-\eta a_{k} - 1)}.
    \end{aligned}
\end{equation}
As a consequence of the minimax theorem, maximizing over \(f\) is equivalent to solving the problem defined in \eqref{eqn: entropic constrained OT general form}. We emphasize that \(f\) is \emph{concave}, allowing one to use routine convex optimization techniques. 

Define \(P = \exp{\big(\eta(-C + \sum_{m}a_{m}D_{m} + x\1^{\top} + \1y^{\top}) - 1\big)}\) as the intermediate matrix corresponding to dual variables \(x,y,a\). We write down the first derivative of the Lyapunov function \(f\) :
\begin{equation}\label{eqn: derivative formula}
\begin{aligned}
    &\nabla_{x} f = r - P \mathbf{1}, \quad \nabla_{y} f = c - P^{\top} \mathbf{1},\\
    &\partial_{a_{k}} f = \exp(-\eta a_{k} - 1) - P \cdot D_{k}, \ \ \forall k \in [K], \\
    &\partial_{a_{l+K}} f = -P \cdot D_{l+K},\ \ \forall l \in [L]\,.
\end{aligned}
\end{equation}
One can use any first-order method on \(f\) using the derivatives provided in \eqref{eqn: derivative formula}. For example, one option with a desirable performance guarantee would be accelerated gradient ascent \citep{nesterov1983method,dvurechensky2018computational}. Overall, the formulation in \eqref{eqn: dual form} shows that the introduction of constraint does not pose a substantial change when compared to the optimization task in the unconstrained case.

\section{Main algorithm}\label{sec: alg}
This section proposes an efficient Sinkhorn-type algorithm in the constrained case. We assume \(K + L = O(1)\) to ensure the efficiency of the proposed approach.
\paragraph{Main idea of the algorithm}
The variational formulation under entropic regularization shows that one can effectively solve for the entropic formulation in \eqref{eqn: entropic constrained OT general form} by solving for the optimization task:
\[
\max_{x \in \R^n,y \in \R^n,a \in \R^{K+L}}f(x,y,a)
\]
for $f$ given in \eqref{eqn: dual form}. We define a Sinkhorn-type algorithm by introducing three iteration steps:
\begin{enumerate}
    \item (\(x\) update) \(x \gets \argmax_{\tilde{x}}f(\tilde{x},y,a)\),
    \item (\(y\) update) \(y \gets \argmax_{\tilde{y}}f(x,\tilde{y},a)\),
    \item (\(a\) update) \(a, t \gets \argmax_{\tilde{a}, \tilde{t}}f(x + \tilde{t}\mathbf{1},y,\tilde{a})\)
\end{enumerate}

The proposed alternating update approach is summarized in Algorithm \ref{alg:1}.

\paragraph{Implementation detail}
Let \((x,y, a)\) be the current dual variables and define \[P = \exp{\left(\eta(-C + \sum_{m}a_{m}D_{m} + x\1^{\top} + \1y^{\top}) - 1\right)}\] as the current intermediate transport plan defined by these dual variables.
By \eqref{eqn: derivative formula}, one has \(\nabla_{x} f = 0 \iff P\mathbf{1} = r\) and \(\nabla_{y} f = 0 \iff P^{\top}\mathbf{1} = c\). Thus, the \(x,y\) update steps amount to row/column scaling of the matrix \(P\), which is identical to the Sinkhorn algorithm. 

The \(a\) update step constitutes the main novelty of our algorithm in the constrained case. In particular, the inclusion of the \(t\) variable improves numerical stability by enforcing a normalization condition: Suppose \(P\) is the intermediate transport plan formed by the dual variable \((x + t \mathbf{1}, y, a)\), then 
\begin{equation}\label{eqn: t optimality}
    \partial_{t} f(x + t \mathbf{1}, y, a) = 0 \iff \sum_{ij}P_{ij} = \sum_{i}r_i,
\end{equation}
which ensures that terms of the form \(M \cdot P\) can be bounded by \(\lVert M \rVert_{\infty}\left(\sum_{i}r_i\right)\). Thus, optimality in the \(t\) variable ensures boundedness in the derivatives of \(f\) such as those in \eqref{eqn: derivative formula}.

We propose to use Newton's method for the \(a\) update step. Namely, by directly computing the gradient and the Hessian, \(\nabla_{at}f, \nabla_{at}^2f\), one uses the search direction $(\Delta a, \Delta t) = -\left(\nabla^2_{at} f\right)^{-1}\nabla_{at} f$. The learning rate is obtained through the standard backtracking line search scheme \citep{boyd2004convex}.

\begin{algorithm}
\caption{Sinkhorn-type algorithm under linear constraint}\label{alg:1}
\begin{algorithmic}[1]
\Require $f, x_{\mathrm{init}},y_{\mathrm{init}},a_{\mathrm{init}}, N, i = 0, \epsilon > 0$

\State $(x,y, a) \gets (x_{\mathrm{init}}, y_{\mathrm{init}}, a_{\mathrm{init}}) $

\While{$i < N$} 
\State $i \gets i + 1$

\State \texttt{\# Row\&Column scaling step}
    \State \(P \gets \exp{\left(\eta(-C + \sum_m a_mD_m + x\1^{\top} + \1y^{\top}) - 1\right)}\)
    \State \( x \gets x + \left(\log(r) - \log(P\mathbf{1} ) \right)/\eta\) 
    \State \(P \gets \exp{\left(\eta(-C + \sum_m a_mD_m + x\1^{\top} + \1y^{\top}) - 1\right)}\)
    \State \( y \gets y + \left(\log(c) - \log(P^{\top}\mathbf{1} ) \right)/\eta\)
\State \texttt{\# Constraint dual update step}
\State \(a, t \gets \argmax_{\tilde{a}, \tilde{t}}f(x + \tilde{t} \mathbf{1},y,\tilde{a})\)
\State \(x \gets x + t\mathbf{1}\)

\EndWhile

\State Output dual variables $(x,y,a)$.
\end{algorithmic}
\end{algorithm}

\paragraph{Complexity analysis of Algorithm \ref{alg:1}}
The row/column scaling step is identical to the Sinkhorn algorithm and thus costs \(O(n^2)\) per iteration. For the constraint dual update step, the computation cost is dominated by the calculation of \(\nabla_{at}^2f\). For the formula of the Hessian, one can refer to \eqref{eqn: Hessian formula} in Appendix \ref{appendix: SNS et al} and see that the cost is dominated by the computation of \(\sum_{ij}P_{ij}\left(D_{m}\right)_{ij}\left(D_{m'}\right)_{ij}\) for \(m, m' = 1, \ldots, K + L\), which leads to a computation complexity of \(O(N_a(K+L)^2n^2))\), where \(N_a\) is the number of iterations taken in the Newton step. The computation complexity in the constraint dual update step is \(O(n^2)\) for two reasons. First, this algorithm assumes \(K + L = O(1)\). Second, Newton's method enjoys super-exponential convergence practically \citep{boyd2004convex}, and thus one can set \(N_a = O(1)\), which is usually sufficient to reach machine precision.

\subsection{Acceleration techniques}
\paragraph{Entropy regularization scheduling} An important feature of the Sinkhorn algorithm is that the iteration complexity heavily depends on the entropy regularization term \(\eta\), the tuning of which plays a significant part in practical performance. To that end, one can aid acceleration by using the doubling entropy regularization scheduling technique introduced in \citep{chen2023exponential}. For a desired entropy regularization value $\eta_{\mathrm{final}}$, we take an initial regularization strength \(\eta_{\mathrm{init}}\) and take $N_{\eta} = \ceil{\mathrm{log}_2(\eta_{\mathrm{final}}/\eta_{\mathrm{init}})}$. Then, one defines successively doubling regularization levels \(\eta_0, \ldots, \eta_{N_{\eta}}\) so that \(\eta_{\mathrm{init}} = \eta_{0} < \ldots < \eta_{N_{\eta}} = \eta_{\mathrm{final}}\) and \(\eta_{i}= 2\eta_{i-1}\) for \(i = 1, \ldots, N_{\eta} - 1\). For each step \(i\), one runs the subroutine in Algorithm \ref{alg:1} at \(\eta_{i-1}\), and the obtained dual variable is used as the initialization when calling the subroutine for \(\eta_{i}\).
\vspace{-0.5em}
\paragraph{Second-order acceleration through sparsity}
We further accelerate the Sinkhorn-type algorithm with Sinkhorn-Newton-Sparse (SNS) \citep{tang2024accelerating}, a second-order method originally developed for the unconstrained case. Instead of alternating maximization as in Algorithm \ref{alg:1}, one can instead introduce the combined variable \(z = (x, y, a)\), where a naive strategy would be to optimize directly \(f\) through a full Newton update of the type  \(z = z - \alpha(\nabla_{z}^2 f)^{-1} \nabla_z f\), where \(\alpha\) can be obtained through backtracking line search. However, a full Newton step would require an impractical \(O((n+K+L)^3)\) scaling.

The SNS algorithm introduces a practical second-order method through a sparse approximation of the Hessian matrix, and we show how one can extend the algorithm to the constrained OT case. The key observation of SNS is that the Hessian submatrix corresponding to variable \(x, y\) has a special structure:
\begin{equation}
\nabla_{xy}^2 f(x, y, a) = \eta
\begin{bmatrix}
\mathrm{diag}(P\1) & P \\
P^{\top} & \mathrm{diag}(P^{\top}\1)
\end{bmatrix},
\end{equation}
which admits a sparse approximation as long as \(P\) admits a sparse approximation. Moreover, the full Hessian matrix to \(f\) admits a sparse approximation: As \(K + L = O(1)\), the blocks of \(\nabla^2 f\) corresponding to \(\nabla_{xy}\nabla_{a}f\) and \(\nabla_a^2 f\) lead to at most \(O(n)\) nonzero entries to keep track of.

The rationale for approximate sparsity is simple: Under a mild uniqueness assumption, the optimal solution \(P^{\star}\) to the LP in \eqref{eqn: constrained OT general form} has at most \(2n - 1 + K+L = O(n)\) nonzero entries due to the fundamental theorem of linear programming \citep{luenberger1984linear}.
Moreover, the exponential closeness result in Theorem \ref{thm: exp convergence OT} implies that the entries of \(P^{\star}_{\eta}\) decay at a rate of \(\exp{(-\frac{\eta\Delta}{K+1})}\) except at \(2n - 1 + K+L = O(n)\) entries, which proves that approximate sparsity holds in the constrained case as well.

We propose an extended SNS algorithm by sparsifying the \(\nabla_{xy}^2 f\) block in which one keeps only \(O(n)\) nonzero entries. The sparse Newton step is used in combination with Algorithm \ref{alg:1} as a warm start to achieve rapid acceleration. The proposed scheme leads to the same \(O(n^2)\) per-iteration complexity as SNS in the unconstrained case. Implementation details can be found in Algorithm \ref{alg:SNS augmented} in Appendix \ref{appendix: SNS et al}.

\section{Convergence analysis}\label{sec: analysis}
In this section, we present a convergence bound for a modified version of the proposed Sinkhorn-type algorithm in terms of the stationarity condition. Similar to the proof in \citep{altschuler2017near} for the unconstrained case, our proof strategy relies on characterizing convergence through the Lyapunov function \(f\) introduced in \eqref{eqn: dual form}. For the rest of this section, we assume \(\sum_{i}r_i = \sum_j c_j = 1\).

The following Theorem characterizes the improvement of the Lyapunov function from one step of the proposed algorithm (proof is in Appendix \ref{appendix: proof of OT poly convergence}):
\begin{theorem}\label{thm: poly convergence OT}
    Let \((x, y, a)\) be the current dual variable which has undergone at least one update step, and let \(P\) be the associated transport matrix, i.e., \[P = \exp{\left(\eta(-C + \sum_{m}a_{m}D_{m} + x\1^{\top} + \1y^{\top}) - 1\right)}.\] 

    For the \(x\) update step satisfying \(x' = \argmax_{\tilde{x}}f(\tilde{x},y,a)\), one has
    \begin{equation}\label{eqn: x step result}
        f(x',y,a) - f(x,y,a) = \KL{r}{P\mathbf{1}},
    \end{equation}
    where \(\KL{\cdot}{\cdot}\) stands for the Kullback-Leibler divergence.
    
    For the \(y\) update step satisfying \(y' = \argmax_{\tilde{y}}f(x,\tilde{y},a)\), one has
    \begin{equation}\label{eqn: y step result}
        f(x,y',a) - f(x,y,a) = \KL{c}{P^{\top}\mathbf{1}}.
    \end{equation}
    Consider the \(a\) update step satisfying \((a',t') = \argmax_{\tilde{a},\tilde{t}}f(x + \tilde{t} \mathbf{1}, y, \tilde{a})\). For \(m = 1, \ldots ,K + L\), define \(d_m\) as the optimality condition corresponding to \(a_m\), i.e.
    \begin{equation}\label{eqn: def of constraint distance}
        \begin{aligned}
            &d_{k} := \exp(-\eta a_{k} - 1) - P \cdot D_{k}, & k = 1, \ldots, K
            \\
            &d_{l + K} := - P \cdot D_{l+K}, &l = 1, \ldots, L.
        \end{aligned}
    \end{equation}
    Then, for \(c_{d} := \max_{m \in [K+L]} \lVert D_m \rVert_{\infty}\), one has
    \begin{equation}\label{eqn: a augmented step result}
    \begin{aligned}
        &f(x + t'\mathbf{1},y,a') - f(x,y,a) \\\geq &\sum_{k=1}^{K}\lvert  d_k \rvert \min\left(
    \frac{1}{8\eta}
    ,
    \frac{\lvert  d_k \rvert}{8\eta c_d + 4\eta(K+L)c_d^2}
    \right)\\
    + 
    &\sum_{l=1}^{L}\frac{d_{l+K}^2}{2\eta (K+L)c_d^2}.
    \end{aligned}
    \end{equation}
\end{theorem}

We now prove a bound on the first-order stationarity condition.
As one is performing alternating optimization for more than two variables, the convergence proof requires a modification to Algorithm \ref{alg:1}, in which one examines the stationary condition on each variable and chooses the update step greedily. The following Theorem characterizes the approximate stationarity of the greedy version of Algorithm \ref{alg:1} (proof is in Appendix \ref{appendix: proof of OT linear convergence}):
\begin{theorem}\label{thm: OT linear convergence}
    Let \((x,y,a)\) be the current variable, let \(P\) be the associated transport matrix, and let \(d_{m}\) be as in \eqref{eqn: def of constraint distance}. Define \(Q_x, Q_y, Q_a\) by
    \begin{equation*}
        \begin{aligned}
            Q_x = &\KL{r}{P\mathbf{1}}, Q_y = \KL{c}{P^{\top}\mathbf{1}},\\
            Q_a = &\sum_{k=1}^{K}\lvert  d_k \rvert \min\left(
    \frac{1}{8\eta}
    ,
    \frac{\lvert  d_k \rvert}{8\eta c_d + 4\eta(K+L)c_d^2}
    \right)\\
    + 
    &\sum_{l=1}^{L}\frac{d_{l+K}^2}{2\eta (K+L)c_d^2},
        \end{aligned}
    \end{equation*}
    which are respectively the right-hand-side of \eqref{eqn: x step result}, \eqref{eqn: y step result}, and \eqref{eqn: a augmented step result}. 
    
    Consider a greedy version of Algorithm \ref{alg:1}, in which only one update step is performed at each iteration, and an \(x\) (resp. \(y\), \(a\)) update step is chosen if \(Q_x\) (resp. \(Q_y\), \(Q_a\)) is the largest among \((Q_x, Q_y, Q_a)\). Let \(c_d\) be as in Theorem \ref{thm: poly convergence OT}, and define \(c_{g}\) as the gap term in the Lyapunov function \(f\) at initialization, i.e.
    \[
    c_g = \max_{\tilde{x},\tilde{y},\tilde{a}}f(\tilde{x},\tilde{y},\tilde{a}) - f(x_{\mathrm{init}},y_{\mathrm{init}},a_{\mathrm{init}}).
    \]
    Let \((x^u, y^u, a^u)\) be the dual variable after \(u\) iterations and let \(P_u\) be the associated intermediate transport plan. Then, the greedy algorithm outputs dual variables $(x, y, a)$ which satisfies the below approximate stationarity condition in \(O(c_g \epsilon^{-2})\) iterations:
    \begin{equation}\label{eqn: greedy main result}
    \lVert \nabla f(x, y, a) \rVert_1 \leq \epsilon.
    \end{equation}
    Moreover, the matrix \(P\) associated with the outputted dual variable satisfies the following condition:
    \begin{equation}\label{eqn: greedy main result 2}
    \begin{aligned}
        \epsilon \geq &\lVert P\mathbf{1} - r \rVert_{1} + \lVert P^{\top}\mathbf{1} - c \rVert_{1}\\
        +
        &\sum_{k=1}^{K}\lvert \min\left(P\cdot D_k, 0\right)  \rvert + \sum_{l=1}^{L}\lvert P\cdot D_{l+K}  \rvert,
    \end{aligned}
    \end{equation}
    which shows that \(P\) approximately satisfies the linear constraints in \eqref{eqn: constrained OT general form}.

\end{theorem}

\section{Numerical experiments}
\begin{figure*}
    \centering
       \includegraphics[width = 0.9\textwidth, trim={0 0.4cm 0 0},clip]{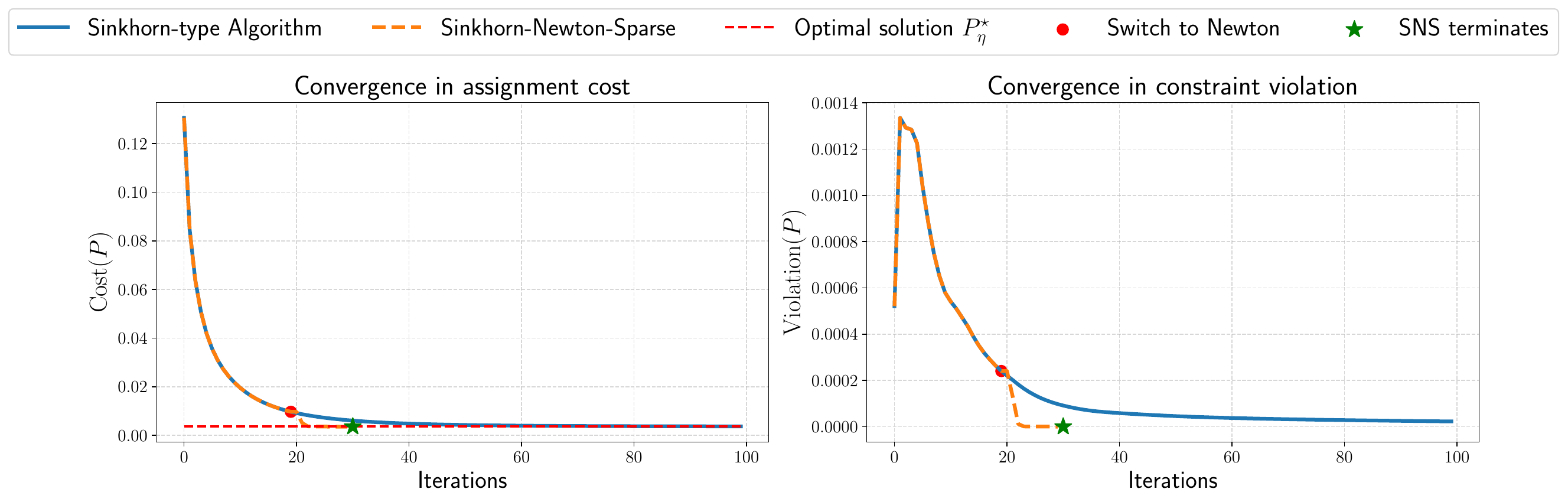}
        \caption{Random assignment problem. Plot of the proposed Sinkhorn-type algorithm in terms of assignment cost and constraint violation. Specifically, constraint violation is defined by $\mathrm{Violation}(P) = |\min(0, D_I \cdot \mathrm{Round}(P,\mathcal{U}_{r, c}))| + |D_E \cdot \mathrm{Round}(P,\mathcal{U}_{r, c})|$.}
        \label{fig:combined_plots_linear_assignment}
        \includegraphics[width = 0.9\textwidth, trim={0 0.4cm 0 1.8cm},clip]{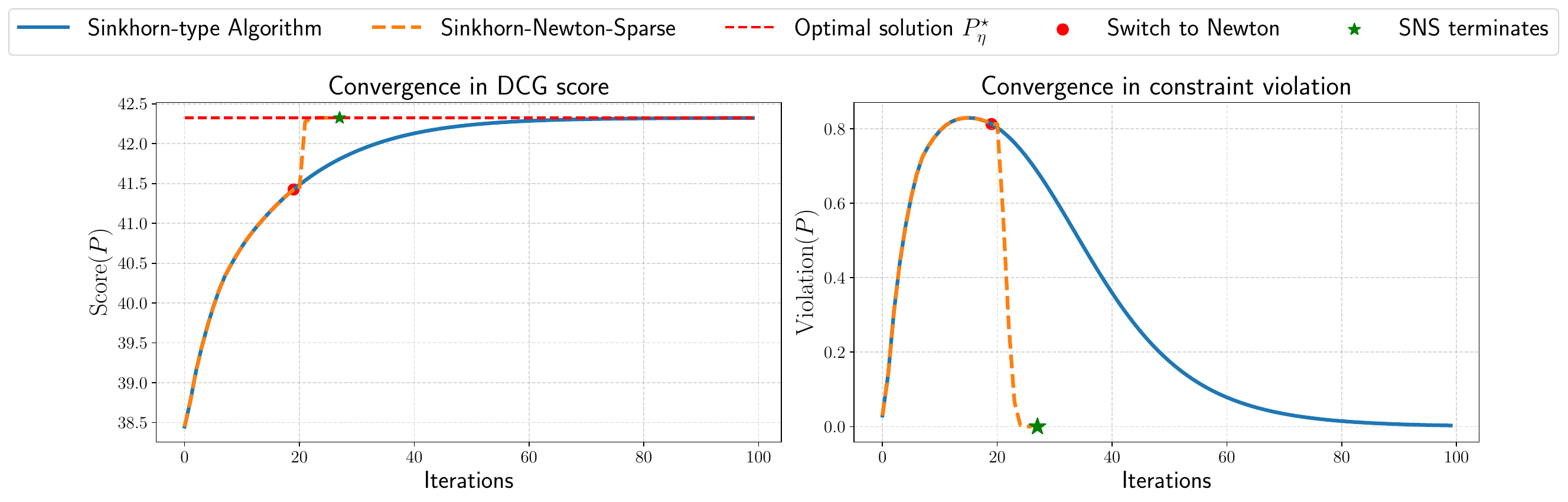}
        \caption{Ranking under constraints. Plot of the proposed Sinkhorn-type algorithm in terms of DCG score and constraint violation. Specifically, constraint violation is defined by $\mathrm{Violation}(P) = |\min(0, D_I \cdot \mathrm{Round}(P,\mathcal{U}_{r, c}))| + |D_E \cdot \mathrm{Round}(P,\mathcal{U}_{r, c})|$.}
        \label{fig:combined_plots_dcg}
\end{figure*}
We conduct numerical experiments to showcase the performance of the proposed Sinkhorn-type algorithm and its acceleration techniques. Let \(\mathcal{U}_{r, c}\) be the set of transport matrices from \(r\) to \(c\). We use the rounding algorithm in \citet{altschuler2017near} to obtain projection into \(\mathcal{U}_{r, c}\). The performance is evaluated through the cost and constraint violation of the transport matrix one obtains through rounding. Specifically, one first uses the
dual variable \((x, y, a)\) to form the intermediate matrix \(P\). Then, one uses the rounding algorithm to get a transport matrix, denoted \(\mathrm{Round}(P, \mathcal{U}_{r, c})\). The cost or score of the transport is evaluated as 
\begin{equation*}
    \begin{aligned}
        &\mathrm{Cost}(P) = C \cdot \mathrm{Round}(P, \mathcal{U}_{r, c}),\\
        &\mathrm{Score}(P) = (-C) \cdot \mathrm{Round}(P, \mathcal{U}_{r, c}).
    \end{aligned}
\end{equation*}
For constraint violation, we use the following metric:
\[\begin{aligned}
    \mathrm{Violation}(P) = 
&\sum_{k=1}^{K}\lvert \min\left(\mathrm{Round}(P, \mathcal{U}_{r, c})\cdot D_k, 0\right)  \rvert\\ + &\sum_{l=1}^{L}\lvert \mathrm{Round}(P, \mathcal{U}_{r, c})\cdot D_{l+K}  \rvert.
\end{aligned}\]
In Proposition \ref{prop: constraint violation bound} in Appendix \ref{appendix: proof of OT linear convergence}, we give an upper bound on \(\mathrm{Violation}(P)\) when one runs the greedy version of Algorithm \ref{alg:1} in Theorem \ref{thm: OT linear convergence}. 
\paragraph{Random assignment problem under constraints}
In the first numerical test, we consider the random assignment problem \citep{mezard1987solution,steele1997probability,aldous2001zeta} with additional inequality and equality constraints. In this setting, we set the problem size of \(n = 500\) and an entropy regularization of \(\eta = 1200\). The source and target vectors are \(c = r = \frac{1}{n}\mathbf{1}\). We consider three \(n \times n\) matrices \(C, D_{I}, D_{E}\), which respectively encode the cost, the inequality constraint, and the equality constraint. We generate the entries of \(C, D_{I}, D_{E}\) by i.i.d. random variables following the distribution \(\mathrm{Unif}([0,1])\). We then set two threshold variables \(t_{I}, t_{E}\) and consider the following optimal transport task:
\begin{equation}\label{eqn: random assignment example}
\begin{aligned}
    &\min_{
\substack{P: P\1=r,\, P^{\top}\1 = c}} &C\cdot P,\\
  &\text{subject to} &D_{I} \cdot P \leq t_{I}, \\
  &  &D_{E} \cdot P = t_E.
\end{aligned}
\end{equation}
The conversion of \eqref{eqn: random assignment example} to the general form (see \eqref{eqn: constrained OT general form}) can be performed by the procedure detailed in Section \ref{sec: variational}. 
Subsequently, we perform the proposed modified Sinkhorn-type algorithm detailed in Algorithm \ref{alg:1}, as well as the accelerated Sinkhorn-Newton-Sparse algorithm in Algorithm \ref{alg:SNS augmented} (see Appendix \ref{appendix: SNS et al}). We set the threshold parameters to be \(t_{I} = t_{E} = \frac{1}{2}\), and it has been verified that the tested instances of \eqref{eqn: random assignment example} is feasible. We test the performance for Algorithm \ref{alg:1} and Algorithm \ref{alg:SNS augmented}. For Algorithm \ref{alg:SNS augmented}, we use \(N_1=20\) Sinkhorn steps for initialization.

In Figure \ref{fig:combined_plots_linear_assignment}, we plot the performance of Algorithm \ref{alg:1} and Algorithm \ref{alg:SNS augmented}.
One can see that the proposed modified Sinkhorn-type algorithm quickly converges to an approximately optimal solution that approximately satisfies the additional constraints. Algorithm \ref{alg:SNS augmented} achieves similar performance under a vastly smaller number of iterations. We remark that Algorithm \ref{alg:SNS augmented} in this case has converged to the optimal solution \(P_{\eta}^{\star}\) with machine accuracy in \(N_1 + N_2 = 25\) iterations. We note that \(P_{\eta}^{\star}\) is obtained by using Newton's method, i.e., running Algorithm \ref{alg:SNS augmented} without sparsification, and the stationarity condition has been verified to be satisfied under machine-epsilon accuracy. In Appendix \ref{appendix: empirical exponential convergence}, we show that Algorithm \ref{alg:1} and Algorithm \ref{alg:SNS augmented} both converge rapidly to the entropic optimal solution \(P_{\eta}^{\star}\).

\paragraph{Ranking under constraints}
In the second numerical example, we consider cost matrices associated with evaluation metrics in ranking problems \citep{liu2009learning,manning2009introduction}. The problem of ranking can be naturally perceived as a transport problem where the transport matrix is analogous to the permutations of items \citep{mena2018learning}. All linear additive ranking metrics, such as precision, recall, and discounted cumulative gain (DCG), are linear functions of the permutation matrix. For example, the DCG metric measures the quality of a ranking \(\sigma: [n] \to [n]\). For a relevance vector \(g \in \R^{n}\) and a discount vector \(v \in \R^{n}\), a DCG score is defined by the following:
\begin{equation*}
    \mathrm{DCG}_{v, g}(\sigma) = \sum_{i = 1}^{n}g_{\sigma(i)}v_{i} =(gv^{\top}) \cdot P 
\end{equation*}
where \(P\) is the permutation matrix associated with \(\sigma\). Thus, optimizing the DCG metric corresponds to an OT instance through the relaxation of the permutation matrix within the Birkhoff polytope.

In practical problems such as e-commerce ranking, there are typically multiple relevance labels in the form of different attributes of an item. Motivated by this, we consider the following constrained OT problem:
\begin{equation}\label{eqn: random dcg example}
\begin{aligned}
    &\max_{
\substack{P: P\1=\mathbf{1},\, P^{\top}\1 = \mathbf{1}}} &(g_{c} v^{\top})\cdot P,\\
  &\text{subject to} &(g_{I} v^{\top}) \cdot P \geq t_{I}, \\
  &  &(g_{E} v^{\top}) \cdot P = t_E,
\end{aligned}
\end{equation}
where \(g_{c}, g_{I}, g_{E}\) are three relevance vectors associated with the cost, the inequality constraint, and the equality constraint, respectively. In accordance with the information retrieval literature, we consider the discount vector \(v\) with \(v_{i} = \frac{1}{\log_{2}(i+1)}\). The entries of \(g_{c}, g_{I}, g_{E}\) are i.i.d. entries simulated from the Rademacher distribution. In \eqref{eqn: random dcg example}, we choose the two threshold variables as \(t_{I} = \frac{1}{n}D_{I} \cdot \mathbf{1}_{n\times n}, t_{E} = \frac{1}{n}D_{E} \cdot \mathbf{1}_{n\times n}\), so that the feasibility set is guaranteed to be non-empty.

In accordance with \eqref{eqn: random assignment example}, we consider a problem size of \(n = 500\) and an entropy regularization of \(\eta = \frac{1200}{500} = 2.4\). 
The performance is plotted in Figure \ref{fig:combined_plots_dcg}, which shows that both Algorithm \ref{alg:1} and Algorithm \ref{alg:SNS augmented} can quickly converge to an approximately optimal solution that approximately satisfies the additional constraints. Furthermore, Algorithm \ref{alg:SNS augmented} in this case has converged to the optimal solution \(P_{\eta}^{\star}\) with machine-epsilon accuracy in under \(25\) iterations. Appendix \ref{appendix: empirical exponential convergence} shows rapid convergence to \(P_{\eta}^{\star}\) for both algorithms.

\paragraph{Pareto front for geometric transport problem under Euclidean distance and Manhattan distance}
In the third numerical example, we consider the trade-off between two transport costs of a geometric nature. In this experiment, we run the task of Pareto front profiling between the Manhattan distance cost \(c_1(x, y) = \lVert x- y \rVert_1\) and the Euclidean distance transport cost \(c_2(x, y) = \lVert x- y \rVert_2^2\). As a byproduct, one obtains an interpolation between the Wasserstein \(W_1\) transport plan under the \(l_1\) distance and the \(W_2\) transport plan under the \(l_2\) distance. Let \(C_1, C_2\) be the cost matrix associated with the Manhattan distance and the Euclidean distance \citep{villani2009optimal}. To do so, we consider the following optimization task:
\begin{equation}\label{eqn: Manhattan vs Euclidean}
\begin{aligned}
    &\min_{
\substack{P: P\1=r,\, P^{\top}\1 = c}} &C_1\cdot P,\\
  &\text{subject to} &C_{2} \cdot P \leq t^2,
\end{aligned}
\end{equation}
where \(t\) is set so that the feasibility set is not empty. We let \(t_{\mathrm{min}}\) be the \(W_2\) distance between the source and target vectors, and let \(t_{\mathrm{max}}\) be the Euclidean transport cost of a transport plan which minimizes Manhattan distance transport cost. We remark that \(t_{\mathrm{min}}\) and \(t_{\mathrm{max}}\) can be obtained through conventional OT algorithms. By tracing the value of \(t\) in \([t_{\mathrm{min}}, t_{\mathrm{max}}]\) and solving \eqref{eqn: Manhattan vs Euclidean} through entropic regularization, one can effectively obtain a Pareto front between Euclidean distance transport cost and Manhattan distance transport cost. 

Similar to previous cases, Algorithm \ref{alg:SNS augmented} is able to reach optimal solution \(P_{\eta}^{\star}\) within machine-epsilon accuracy. Thus, we run multiple instances of Algorithm \ref{alg:SNS augmented} for \(t \in [t_{\mathrm{min}}, t_{\mathrm{max}}]\) and \(\eta = 10, 100, 1000\). Then, we plot the Pareto front formed by \(P_{\eta}^{\star}\) for every choice of \(\eta\). The true Pareto front can be obtained through running Algorithm \ref{alg:SNS+annealing} in Appendix \ref{appendix: SNS et al} for \(\eta_{\mathrm{target}} = 8*10^4\), and convergence is verified through checking the first order derivative reaches machine accuracy. 

Similar to \citet{cuturi2013sinkhorn}, we illustrate the procedure through optimal transport on the MNIST dataset. We pick two images, which are converted to a vector of intensities on the $28 \times 28$ pixel grid and normalized to sum to 1. The entry corresponding to the \((i_1,i_2)\)-th pixel is conceptualized as the point \((i_1/28, i_2/28) \in \mathbb{R}^{2}\). In Figure \ref{fig:combined_plots_Pareto_front}, we see that the Pareto front formed by increasing \(\eta\) indeed converges to the true Pareto front.
\begin{figure}
    \centering
    \includegraphics[width = 0.45\textwidth, trim={0 0.2cm 0 0},clip]{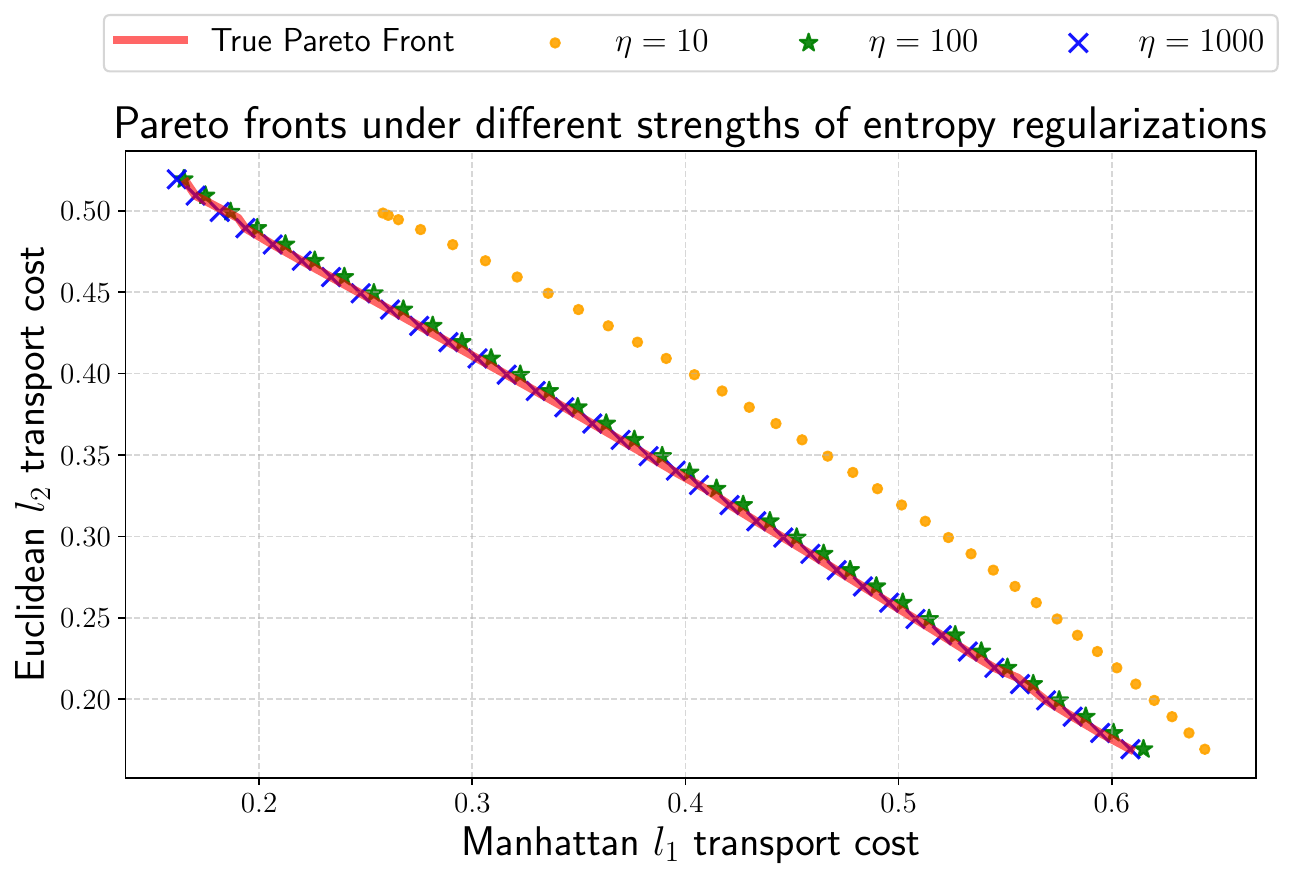}
    \caption{Pareto front profiling of Euclidean distance versus Manhattan distance.}
    \label{fig:combined_plots_Pareto_front}
\end{figure}

\section{Conclusion}
We introduce an entropic formulation of optimal transport with a combination of additional equality and inequality constraints. We propose a Sinkhorn-type algorithm that has a novel constraint dual variable update step. We provide preliminary results on the approximation error of the entropic formulation and the convergence of the Sinkhorn-type algorithm, while more work remains to be done for the convergence property of Algorithm \ref{alg:1} in terms of convergence to the optimal entropic solution \(P_{\eta}^{\star}\). Moreover, we provide a detailed discussion on improving the proposed Sinkhorn-type algorithm with acceleration techniques, especially Sinkhorn-Newton-Sparse and entropy regularization scheduling. The proposed work enables one to obtain approximately optimal solutions in more complicated OT instances efficiently. We contend that this work has the potential to be a vital tool in the field of optimal transport under constraint and in attracting the use of constrained optimal transport in machine learning.

\bibliography{references}
\bibliographystyle{plainnat}

\newpage
\appendix
\onecolumn
\section{Proof of Theorem \ref{thm: exp convergence OT}}\label{appendix: proof of OT exp convergence}

\begin{definition}\label{defn: notation for LP vertex}
    Let \(\mathcal{S}\) be the constraint set defined in \eqref{eqn: constrained OT general form}. Define \(\mathcal{P}\) as the polyhedron formed by the transport matrix, i.e. \[\mathcal{P} := \{P \mid P\1=r , P^{\top}\1 = c, P \geq 0, P \in \mathcal{S}\}.\] The symbol \(\mathcal{V}\) denotes the set of vertices of \(\mathcal{P}\). The symbol \(\mathcal{O}\) 
    stands for the set of optimal vertex solutions, i.e.
    \begin{equation}\label{eqn: LP}
        \mathcal{O} := \argmin_{P \in \mathcal{V}} C\cdot P.
    \end{equation}
    The symbol \(\Delta\) denotes the vertex optimality gap \[\Delta = \min_{Q \in \mathcal{V} - \mathcal{O}}Q \cdot C - \min_{P \in \mathcal{O}}P \cdot C.\]
\end{definition}

We can now finish the proof. 
\begin{proof}
    This convergence result is mainly due to the application of Corollary 9 in \citet{weed2018explicit} to this case. We define another polyhedron \(\mathcal{Q}\) as follows:
    \[\mathcal{Q} := \{(P, s) \mid P \in \mathcal{P}, \forall k \in [K], s_k = D_k \cdot P\}.\]
    Let \(P_{\eta}^{\star}\) be as defined in the statement, and for \(k = 1, \ldots, K\) we define \(s_{\eta; k}^{\star} = D_k \cdot P_{\eta}^{\star}\). We use \(R_1\) and \(R_{H}\) to denote the \(l_{1}\) and entropic radius of \(\mathcal{Q}\) in the sense defined in \citet{weed2018explicit}.
    It is easy to see that for \(R_1\) one has
    \[
    1 \leq R_{1} = 1 + \max_{P \in \mathcal{P} }\sum_{k = 1}^{K}P \cdot D_k \leq 1 + K,
    \]
    where the second inequality used Holder's inequality and the assumption that \(\lVert D_{k} \rVert_{\infty} \leq 1\) 

    For \(R_{H}\), one similarly has 
    \begin{equation*}
        \begin{aligned}
            R_{H} = &\max_{(s, P), (s', P') \in \mathcal{Q}}\sum_{ij}\left(p_{ij}\log(p_{ij}) - p'_{ij}\log(p'_{ij}) \right) + \sum_{k}\left(s_{k}\log(s_{k}) - s'_{k}\log(s'_{k})\right)\\
            \leq &\left(\max_{P, P' \in \mathcal{P}}\sum_{ij}\left(p_{ij}\log(p_{ij}) - p'_{ij}\log(p'_{ij}) \right)\right) + 
            \left(
            \max_{P, P' \in \mathcal{P}}\sum_{k}(P \cdot D_k)\log(P \cdot D_k) - (P' \cdot D_k)\log(P' \cdot D_k) \right)\\
            \leq & \log(n^2) + K/e,
        \end{aligned}
    \end{equation*}
    where the second equality holds because \(H(P)\in [0, \log{(n^2)}]\) and \(P \cdot D_k\log(P \cdot D_k) \in [-1/e, 0]\). 

    For \(\eta \geq \frac{(K+1)(1 + \ln (4 n^2 (K+1))}{\Delta} > \frac{R_1 + R_H}{\Delta}\), one has
    \begin{align*}
        \lVert P^{\star} - P^{\star}_{\eta} \rVert_{1} \leq  &\lVert (P^{\star}, s^{\star}) - (P^{\star}_{\eta}, s_{\eta}^{\star}) \rVert_{1} \\
        \leq &2R_1 \exp\left(-\eta\frac{\Delta}{R_1} + 1 + \frac{R_H}{R_1}\right) \\
        = & 2R_1 \exp\left(\frac{R_H - \eta\Delta}{R_1} + 1\right)\\
        \leq &2(K+1) \exp\left(\frac{R_H - \eta\Delta}{K+1} + 1\right)\\
        = &2(K+1) \exp\left(\frac{2\log(n) + K/e - \eta\Delta}{K+1} + 1\right)\\
        \leq &8n^{\frac2{K+1}}(K+1)\exp\left(-\eta\frac{\Delta}{K+1}\right),
    \end{align*}
    where the third inequality is because \(R_{H} - \eta \Delta \leq 0\), and the last inequality holds because \(\exp(\frac{K/e}{K+1}+1) \leq \exp(1+1/e) \leq 4\).
\end{proof}

\section{Proof of Theorem \ref{thm: poly convergence OT}}\label{appendix: proof of OT poly convergence}

For legibility, in what follows, we introduce the symbols \(\hat{x}, \hat{y}, \hat{a}\) used exclusively for dummy variables.

\begin{proof}
    For the reader's convenience, we list the explicit form of the Lyapunov function \(f\) here:
    \begin{equation*}
            f(\hat{x},\hat{y},\hat{a})
            =-\frac{1}{\eta}\sum_{ij}\exp{\left(\eta(-C_{ij} + \sum_{m = 1}^{L + K}\hat{a}_{m}\left(D_{m}\right)_{ij}+ \hat{x}_i + \hat{y}_{j}) - 1 \right)}
            + \sum_{i}\hat{x}_{i}r_{i} + \sum_{j}\hat{y}_{j}c_{j} - \frac{1}{\eta}\sum_{k = 1}^{K} {\exp(-\eta \hat{a}_{k} - 1)}.
    \end{equation*}
    First we prove \(\sum_{ij}P_{ij} = 1\). By assumption, the dual variable \((x, y, a)\) has gone through at least one update step. If the \(x\) update step has been last performed, then one has \(\nabla_x f = 0\), which implies \(\sum_{ij} P_{ij} = \sum_{i}r_i = 1\). If the \(y\) update step has been last performed, then one likewise has \(\nabla_y f = 0\) and \(\sum_{ij} P_{ij} = \sum_{j}c_j = 1\). If the \(a\) update step has been performed, then the optimality in the \(t\) variable as shown in \eqref{eqn: t optimality} implies \(\sum_{ij}P_{ij} = 1\). Thus, one has \(\sum_{ij}P_{ij} = 1\) in all the three possible cases, as claimed.
    
    The proof for \eqref{eqn: x step result} and \eqref{eqn: y step result} then largely follows from Lemma 2 of \citep{altschuler2017near}.
    Suppose that an \(x\) update step is performed and \(P'\) is the matrix formed by the dual variable \((x', y, a)\). Due to the optimality of the \(x'\) variable, one has \(\sum_{ij}P'_{ij} = 1\). Thus one has
    \[
    f(x', y, a) - f(x, y, a) = \frac{1}{\eta}(\sum_{ij}P_{ij} - \sum_{ij}P'_{ij}) + \sum_{i}r_i(x'_i - x_i) = \sum_{i}r_i(x'_i - x_i) = \KL{r}{P\mathbf{1}},
    \]
    where the last equality is due to \(r_i(x'_i - x_i) = r_i \log \frac{r_i}{\left(P\mathbf{1}\right)_i}\). The proof for \eqref{eqn: y step result} follows likewise:
    \[
    f(x, y', a) - f(x, y, a) = \frac{1}{\eta}(\sum_{ij}P_{ij} - \sum_{ij}P'_{ij}) + \sum_{j}c_j(y'_j - y_j) = \sum_{j}c_j(y'_j - y_j) = \KL{c}{P^{\top}\mathbf{1}}.
    \]

    We proceed with the proof for \eqref{eqn: a augmented step result}. First, we introduce an augmented Lyapunov function by maximizing over the \(t\) variable, which provides a smoother optimization landscape for subsequent analysis. Define the augmented Lyapunov function \(f_{\mathrm{aug}}\) as
    \[f_{\mathrm{aug}}(\hat{x}, \hat{y}, \hat{a}) := \max_{\tilde{t}}f(\hat{x} + \tilde{t} \mathbf{1}, \hat{y}, \hat{a}).\]
    As \(\sum_{ij}P_{ij} = 1\) implies optimality over the \(t\) variable by \eqref{eqn: t optimality}, one has \(f(x, y, a) = f_{\mathrm{aug}}(x, y, a)\) for the dual variable \((x, y, a)\). After the augmented \(a\) update step, one has by definition that \(f(x + t'\mathbf{1}, y, a') = f_{\mathrm{aug}}(x, y, a')\). Thus, one can simplify the left-hand side of \eqref{eqn: a augmented step result} by the following equation:
    \[
    f(x + t'\mathbf{1}, y, a') - f(x, y, a)=
    f_{\mathrm{aug}}(x, y, a') - f_{\mathrm{aug}}(x, y, a) = \max_{\tilde a}f_{\mathrm{aug}}(x, y, \tilde a) - f_{\mathrm{aug}}(x, y, a).
    \]

    We then derive the formula for \(f_{\mathrm{aug}}\). By direct calculation, one has 
    \begin{align*}
        &f(\hat{x} + t\mathbf{1} ,\hat{y},\hat{a})\\
        =
        &-\frac{1}{\eta}\sum_{ij}\exp{\left(\eta(-C_{ij} + \sum_{m = 1}^{L + K}\hat{a}_{m}\left(D_{m}\right)_{ij}+ \hat{x}_i + \hat{y}_{j} + t) - 1 \right)}
        + \sum_{i}\hat{x}_{i}r_{i} + \sum_{j}\hat{y}_{j}c_{j} + t - \frac{1}{\eta}\sum_{k = 1}^{K} {\exp(-\eta \hat{a}_{k} - 1)}.
    \end{align*}
    Let \(t^{\star} = \argmax_{\tilde{t}}f(\hat{x} + \tilde{t}\mathbf{1} ,\hat{y},\hat{a})\).
    Under the optimality condition \(\partial_{t} f(\hat{x} + t^{\star}\mathbf{1} ,\hat{y},\hat{a}) = 0\), one has \[
    \sum_{ij}\exp{\left(\eta(-C_{ij} + \sum_{m = 1}^{L + K}\hat{a}_{m}\left(D_{m}\right)_{ij}+ \hat{x}_i + \hat{y}_{j} + t^{\star}) - 1 \right)} = 1.
    \]
    Taking log over the above equation, one has the following result for \(t^{\star}\):
    \begin{equation}\label{eqn: t optimality in LSE}
        1 - \eta t^{\star} = \LSE{\eta(-C + \sum_{m}\hat{a}_mD_m + \hat{x} \mathbf{1}^{\top} + \mathbf{1} \hat{y}^{\top})},
    \end{equation}
    where for a matrix \(M\), the function \(\mathrm{LSE}\) defines the log-sum-exponential function \(\LSE{M} = \log\left(\sum_{ij}\exp{\left(m_{ij}\right)}\right)\). Thus one has 
    \begin{align*}
    f_{\mathrm{aug}}(\hat{x},\hat{y},\hat{a}) =& f(\hat{x} + t^{\star}\mathbf{1} ,\hat{y},\hat{a})
        \\
        =&-\frac{1}{\eta} +
        \sum_{i}\hat{x}_{i}r_{i} + \sum_{j}\hat{y}_{j}c_{j} + t^{\star} - \frac{1}{\eta}\sum_{k = 1}^{K} {\exp(-\eta \hat{a}_{k} - 1)}
        \\
        =
        &-\frac{1}{\eta}\LSE{\eta(-C + \sum_{m}\hat{a}_mD_m + \hat{x} \mathbf{1}^{\top} + \mathbf{1} \hat{y}^{\top})}
        + \sum_{i}\hat{x}_{i}r_{i} + \sum_{j}\hat{y}_{j}c_{j} - \frac{1}{\eta}\sum_{k = 1}^{K} {\exp(-\eta \hat{a}_{k} - 1)}.
    \end{align*}
The rest of the proof follows from a perturbational argument on \(f_{\mathrm{aug}}\) around the point \((x,y,a)\), and thus we provide a formula for the derivatives of \(f_{\mathrm{aug}}\). Let \(\hat{P}=\exp{\left(\eta(-C + \sum_{m}\hat{a}_{m}D_{m} + \hat{x}\1^{\top} + \1\hat{y}^{\top}) - 1\right)}\) be the transport plan formed by the dual variable \((\hat{x}, \hat{y}, \hat{a})\). For \(k = 1, \ldots, K\), one has 
\[
\partial_{a_k}f_{\mathrm{aug}}(\hat{x}, \hat{y}, \hat{a})
= 
{\exp(-\eta \hat{a}_{k} - 1)} -D_k \cdot \frac{\hat{P}}{\sum_{ij}\hat{P}_{ij}},
\]
and likewise for \(l = 1, \ldots, L\) one has 
\[
\partial_{a_{l+K}}f_{\mathrm{aug}}(\hat{x}, \hat{y}, \hat{a})
= 
-D_{l+K} \cdot \frac{\hat{P}}{\sum_{ij}\hat{P}_{ij}}.
\]

In terms of the second-order information, by direct calculation, one has
\begin{equation}\label{eqn: f_aug second order}
    \partial_{a_{m}} \partial_{a_{m'}} f_{\mathrm{aug}}(\hat{x}, \hat{y}, \hat{a}) = 
    - \eta(D_m \odot D_{m'})\cdot \frac{\hat{P}}{\sum_{ij}\hat{P}_{ij}}
    +
    \eta \frac{\left(\hat{P} \cdot D_m\right)\left(\hat{P} \cdot D_{m'}\right)}{\left(\sum_{ij}\hat{P}_{ij}\right)^2}
    + 
    s_{mm'},
\end{equation}
where \(\odot\) is the Hadamard element-wise product, and \(s_{mm'}\) is zero except when \(m = m' = k\) for \(k = 1, \ldots, K\), in which case one has \(s_{mm'} = -\eta\exp(-\eta\hat{a}_{k}-1)\). 

We then bound the spectrum of the Hessian matrix. Define
\begin{equation}\label{eqn: def of rmm}
    r_{mm'} = - \eta(D_m \odot D_{m'})\cdot \frac{\hat{P}}{\sum_{ij}\hat{P}_{ij}}
    +
    \eta \frac{\left(\hat{P} \cdot D_m\right)\left(\hat{P} \cdot D_{m'}\right)}{\left(\sum_{ij}\hat{P}_{ij}\right)^2},
\end{equation}
and \(R = [r_{mm'}]_{m, m' = 1}^{K+L}\). Let \((I,J)\) be a random vector supported on \([n] \times [n]\), and moreover let \((I, J)\) follow the multinomial distribution with \(\mathbb{P}\left[(I, J) = (i,j)\right] = \frac{\hat{P}_{ij}}{\sum_{ij}\hat{P}_{ij}}\). Moreover, for \(m = 1, \ldots, K+L\), let \(Y_m\) be the random variable defined by \(Y_m = \left(D_m\right)_{IJ}\). Then, one can directly verify that \(R = -\eta \mathrm{Cov}\left(Y_1, \ldots, Y_{K+L}\right)\), and therefore \(R\) is a negative semi-definite matrix. Let \(S =  [s_{mm'}]_{m, m' = 1}^{K+L}\). The matrix \(S\) is a diagonal matrix with non-positive entries and therefore is likewise negative semi-definite. By the trace bound one has
\begin{align*}
    \lVert R\rVert_{2}
    \leq &-\sum_{m = 1}^{K+L} r_{mm}\\ = &\sum_{m=1}^{K+L} \eta\left((D_m \odot D_{m})\cdot \frac{\hat{P}}{\sum_{ij}\hat{P}_{ij}}
    -
    \left(\frac{\hat{P} \cdot D_m}{\sum_{ij}\hat{P}_{ij}}\right)^2\right)\\
    \leq &\sum_{m=1}^{K+L} \eta\left((D_m \odot D_{m})\cdot \frac{\hat{P}}{\sum_{ij}\hat{P}_{ij}}\right)\\
    \leq &\eta(K+L)c_d^2,
\end{align*}
and therefore one has the spectral bound \(S - \eta(K+L)c_d^2 I \preceq \nabla_{a}^2f_{\mathrm{aug}} \preceq S\).

For simplicity, for the current variable \((x,y,a)\), define a function \(g\) by \(g(\hat{a}) = f_{\mathrm{aug}}(x, y, \hat{a})\).
By Taylor's remainder theorem, for any \(\delta a\), one has
\[
f_{\mathrm{aug}}(x, y, a + \delta a) - f_{\mathrm{aug}}(x, y, a) = g(a + \delta a) - g(a) = \delta a^{\top}\nabla_a g(a) + \frac{1}{2}\delta a^{\top} \nabla^{2}_{a}g(a + \xi \delta a)\delta a,
\]
where \(\xi \in [0, 1]\) is an unknown quantity. Thus, the spectral bound on \(\nabla_{a}^2f_{\mathrm{aug}}\) leads to the following inequality,
\begin{equation}\label{eqn: lower bound a step}
\begin{aligned}
    &f_{\mathrm{aug}}(x, y, a + \delta a) - f_{\mathrm{aug}}(x, y, a) \\
    \geq 
    &
    \delta a^{\top}\nabla_a g(a) - \frac{1}{2}\eta(K+L)c_d^2 \lVert \delta a \rVert_{2}^2 - \frac{1}{2}\sum_{k = 1}^{K}\eta\exp(-\eta a_k-1)\exp(-\eta\xi \delta a_k)\delta a_k^2
    \\
    =
    &\sum_{m=1}^{K+L}\left(\delta a_m d_m - \frac{1}{2}\eta(K+L)c_d^2 \delta a_m^2\right) -\frac{1}{2}\sum_{k = 1}^{K}\eta\exp(-\eta a_k-1)\exp(-\eta\xi \delta a_k)\delta a_k^2
\end{aligned}
\end{equation}
where \(d_m = \partial_{a_m}g(a)\) is due to \(\sum_{ij}P_{ij} = 1\).

Importantly, the lower bound in \eqref{eqn: lower bound a step} is fully separable in terms of the \(\delta a_m\) terms. In what follows, we use \eqref{eqn: lower bound a step} to give a construction for \(\delta a\), which would then give a lower bound for the improvement in the Lyapunov function in the \(a\) update step. 
For \(l = 1, \ldots, L\), one sets \(\delta a_{l+K} =  \frac{d_{l+K}}{
\eta (K+L)c_d^2}\). One then has 
\[
\delta a_{l+K} d_{l+K} - \frac{1}{2}\eta(K+L)c_d^2 \delta a_{l+K}^2 = \frac{d_{l+K}^2}{2\eta (K+L)c_d^2}.
\]
For \(k = 1, \ldots K\), one sets
\begin{equation}\label{eqn: choice of delta a_k}
    \delta a_k = \max\left(-
\frac{\log(2)}{\eta}, \frac{d_k}{
\eta\left(
2\exp(-\eta a_k-1) + (K+L)c_d^2
\right)}
\right).
\end{equation}
We shall prove that the construction of \(\delta a_k\) leads to the following bound:
\begin{align*}
    \delta a_k d_k - \frac{1}{2}\eta(K+L)c_d^2 \delta a_k^2 -\frac{1}{2}\eta\exp(-\eta a_k-1)\exp(-\eta\xi \delta a_k)\delta a_k^2
    \geq
    \lvert  d_k \rvert \min\left(
    \frac{1}{8\eta}
    ,
    \frac{\lvert  d_k \rvert}{8\eta c_d + 4\eta(K+L)c_d^2}
    \right).
\end{align*}

By construction, one has \(\exp(-\eta\xi \delta a_k) \leq 2\), and thus \[-\frac{1}{2}\eta\exp(-\eta a_k-1)\exp(-\eta\xi \delta a_k)\delta a_k^2 \geq -\frac{1}{2}\left(2\eta\exp(-\eta a_k-1)\delta\right) a_k^2.\]

If \(-
\frac{\log(2)}{\eta} \geq \frac{d_k}{
\eta\left(
2\exp(-\eta a_k-1) + (K+L)c_d^2
\right)}\), then one has \(\delta a_k = -\frac{\log(2)}{\eta} \leq 0\) and \(d_k \leq 0\). By \eqref{eqn: choice of delta a_k} one has
\[
\delta a_k\eta\left(
2\exp(-\eta a_k-1) + (K+L)c_d^2
\right) \geq d_k,
\]
and thus multiplying both sides by \(-\frac{1}{2}\delta a_k\), the inequality becomes
\[
- \frac{1}{2}\eta\left(  2\exp(-\eta a_k-1) + (K+L)c_d^2\right)\delta a_k^2 \geq  -\frac{1}{2} \delta a_k d_k.
\]
Thus
\begin{align*}
    &\delta a_k d_k - \frac{1}{2}\eta(K+L)c_d^2 \delta a_k^2 -\frac{1}{2}\eta\exp(-\eta a_k-1)\exp(-\eta\xi \delta a_k)\delta a_k^2 \\
    \geq
    &\delta a_k d_k - \frac{1}{2}\eta\left(  2\exp(-\eta a_k-1) + (K+L)c_d^2\right)\delta a_k^2
    \\
    \geq
    &
    \delta a_k d_k - \frac{1}{2} \delta a_k d_k 
    \\
    = &\frac{1}{2} \delta a_k d_k \\
    = & \lvert d_k \rvert \frac{\log(2)}{2\eta},
\end{align*}
which in particular implies the claimed bound as \(\frac{\log(2)}{2} \geq \frac{1}{8}\).

Otherwise, if \(-
\frac{\log(2)}{\eta} \leq \frac{d_k}{
\eta\left(
2\exp(-\eta a_k-1) + (K+L)c_d^2
\right)}\), then one has \(\delta a_k = \frac{d_k}{
\eta\left(
2\exp(-\eta a_k-1) + (K+L)c_d^2
\right)}\) by construction. Then, one has 
\begin{align*}
    &\delta a_k d_k - \frac{1}{2}\eta(K+L)c_d^2 \delta a_k^2 -\frac{1}{2}\eta\exp(-\eta a_k-1)\exp(-\eta\xi \delta a_k)\delta a_k^2 \\
    \geq
    &\delta a_k d_k - \frac{1}{2}\eta\left(  2\exp(-\eta a_k-1) + (K+L)c_d^2\right)\delta a_k^2
    \\
    =
    &\frac{d_k^2}{2\eta\left(  2\exp(-\eta a_k-1) + (K+L)c_d^2\right)}
    \\
    =&
    \frac{d_k^2}{ 4\eta\exp(-\eta a_k-1) + 2\eta(K+L)c_d^2}
    \\
    \geq&
    \lvert  d_k \rvert\frac{\lvert  d_k \rvert}{ 4\eta\lvert  d_k \rvert + 4\eta c_d + 2\eta(K+L)c_d^2}
    \\
    \geq 
    &
    \lvert  d_k \rvert \min\left(
    \frac{1}{8\eta}
    ,
    \frac{\lvert  d_k \rvert}{8\eta c_d + 4\eta(K+L)c_d^2}
    \right),
\end{align*}
where the third inequality is obtained by applying the mediant inequality, and the second inequality is because
\[
\exp(-\eta a_k-1) \leq \lvert \exp(-\eta a_k-1) - P \cdot D_{k} \rvert + \lvert P \cdot D_{k} \rvert = \lvert  d_k \rvert + \lvert P \cdot D_{k} \rvert \leq \lvert  d_k \rvert + c_d.
\]

The proof for \eqref{eqn: a augmented step result} is by organizing the arranged results:
\begin{align*}
    &f(x + t'\mathbf{1}, y, a') - f(x, y, a)
    \\
    =
    &f_{\mathrm{aug}}(x, y, a') - f_{\mathrm{aug}}(x, y, a) 
    \\
    = &\max_{\tilde{a}}f_{\mathrm{aug}}(x, y, \tilde{a}) - f_{\mathrm{aug}}(x, y, a)
    \\
    \geq 
    &f_{\mathrm{aug}}(x, y, a + \delta a) - f_{\mathrm{aug}}(x, y, a)
    \\
    =& g(a + \delta a) - g(a)
    \\
    = &\delta a^{\top}\nabla_a g(a) + \frac{1}{2}\delta a^{\top} \nabla^{2}_{a}g(a + \xi \delta a)\delta a
    \\
    \geq &\sum_{k=1}^{K}\left(\delta a_k d_k - \frac{1}{2}\eta(K+L)c_d^2 \delta a_k^2 - \frac{1}{2}\eta \exp(-\eta a_k-1)\exp(-\eta\xi \delta a_k)\delta a_k^2\right) \\ 
    &\qquad
    +\sum_{l=1}^{L}\left(\delta a_{l+K} d_{l+K} - \frac{1}{2}\eta(K+L)c_d^2 \delta a_{l+K}^2 \right)\\
    \geq 
    &
    \sum_{k=1}^{K}\lvert  d_k \rvert \min\left(
    \frac{1}{8\eta}
    ,
    \frac{\lvert  d_k \rvert}{8\eta c_d + 4\eta(K+L)c_d^2}
    \right) + \sum_{l=1}^{L}\frac{d_{l+K}^2}{2\eta (K+L)c_d^2}.
\end{align*}

\end{proof}

\begin{remark}
    In the proof for Theorem \ref{thm: poly convergence OT}, it might be advantageous to use an alternative spectral bound when the matrices \(D_m\) have special sparsity structures. In such special cases, it is better to use the Gershgorin circle theorem instead of the trace bound in the numerical treatment for the \(a\) step. 

    One defines
    \begin{equation}\label{eqn: def of rmm1 and rmm2}
        \begin{aligned}
            &r_{1; mm'} = - \eta(D_m \odot D_{m'})\cdot \frac{\hat{P}}{\sum_{ij}\hat{P}_{ij}},\\
            &r_{2; mm'} = 
            \eta \frac{\left(\hat{P} \cdot D_m\right)\left(\hat{P} \cdot D_{m'}\right)}{\left(\sum_{ij}\hat{P}_{ij}\right)^2},
        \end{aligned}
    \end{equation}
    and let \(R_1 = [r_{1; mm'}]_{m,m' = 1}^{K+L}, R_2 = [r_{2; mm'}]_{m,m' = 1}^{K+L}\). Let \(Y_m\) for \(m = 1, \ldots, K+L\) be as in the proof of Theorem \ref{thm: poly convergence OT}. One can directly verify that \(R_{2}\) is positive semi-definite as it is the outer product of the vector \(\sqrt{\eta}\mathbb{E}\left[Y_1, \ldots, Y_{K+L}\right]\) with itself, and thus \(S + R_1 \preceq \nabla_{a}^2f_{\mathrm{aug}}\). By the Gershgorin circle theorem, it follows that \(\lVert R_1 \rVert_{2}\) is bounded by the matrix 1-norm of \(R_1\). In particular, define \(E \in \R^{n \times n}\) as the sum of the constraint matrices up to absolute value, i.e.
    \[
    \left(E\right)_{ij} = \sum_{m = 1}^{K+L}\lvert \left(D_{m}\right)_{ij} \rvert.
    \]
    Then, one has
    \begin{align*}
        \lVert R_1 \rVert_{2}
        \leq &\max_{m} \sum_{m' = 1}^{K+L}\lvert r_{1;mm'} \rvert \\
        = &\eta \max_{m} \sum_{m' = 1}^{K+L}\left| (D_m \odot D_{m'})\cdot \frac{\hat{P}}{\sum_{ij}\hat{P}_{ij}} \right|\\
        = & \eta \max_{m} \sum_{m' = 1}^{K+L}\left| \left(\frac{\hat{P}}{\sum_{ij}\hat{P}_{ij}}  \odot D_{m'} \right)\cdot D_m \right|\\
        \leq & \eta c_d \sum_{m' = 1}^{K+L}\lVert \frac{\hat{P}}{\sum_{ij}\hat{P}_{ij}}  \odot D_{m'}  \rVert_{1}\\
        = & \eta c_d \sum_{m' = 1}^{K+L}\left| \frac{\hat{P}}{\sum_{ij}\hat{P}_{ij}}  \cdot D_{m'}  \right|\\
        \leq & \eta c_d \left(\frac{\hat{P}}{\sum_{ij}\hat{P}_{ij}} \cdot E\right)\\
        \leq & \eta c_d \lVert E \rVert_{\infty},
    \end{align*}
    where the second and third equality is by the property of the Hadamard product, the second and fourth inequality is by Holder's inequality, and the third inequality is by the definition of \(E\). We define \(c_{e} = \lVert E \rVert_{\infty}\), which gives one the following bound in contrast to \eqref{eqn: lower bound a step}:
    \begin{equation}\label{eqn: lower bound a step alternative}
\begin{aligned}
    &f_{\mathrm{aug}}(x, y, a + \delta a) - f_{\mathrm{aug}}(x, y, a) \\
    \geq 
    &
    \delta a^{\top}\nabla_a g(a) - \frac{1}{2}\eta c_{e}c_d \lVert \delta a \rVert_{2}^2 - \frac{1}{2}\sum_{k = 1}^{K}\eta\exp(-\eta a_k-1)\exp(-\eta\xi \delta a_k)\delta a_k^2
    \\
    =
    &\sum_{m=1}^{K+L}\left(\delta a_m d_m - \frac{1}{2}\eta c_{e}c_d \delta a_m^2\right) -\frac{1}{2}\sum_{k = 1}^{K}\eta\exp(-\eta a_k-1)\exp(-\eta\xi \delta a_k)\delta a_k^2.
\end{aligned}
\end{equation}
Using \eqref{eqn: lower bound a step alternative}, one can derive a tighter bound in the conclusion part of Theorem \ref{thm: poly convergence OT} by applying a similar argument in the main statement. Overall, the bound in \eqref{eqn: lower bound a step alternative} is more advantageous when \(c_e\) is significantly smaller than \((K+L)c_d\), and indeed one can see that \(c_e \leq (K+L)c_d\). One situation in which an advantage exists is when the constraints \(D_m\) are themselves sparse, which is the case in capacity constrained OT.

\end{remark}

\section{Proof of Theorem \ref{thm: OT linear convergence} }
\label{appendix: proof of OT linear convergence}

\begin{proof}
    For any iteration index \(u\), let \(Q^u_x, Q^u_y, Q^u_a\) be the quantity as defined by the Theorem for \(P_u\) in place of \(P\). Similarly we define \(d_m^u\) by \eqref{eqn: def of constraint distance} with \(P_u\) in place of \(P\). One then has by construction that
    \begin{equation}\label{eqn: decrease per greedy step}
        f(x^{u+1}, y^{u+1}, a^{u+1}) - f(x^{u}, y^{u}, a^{u})\geq \frac{1}{3}\left(Q_x^u + Q_y^u + Q_a^u\right).
    \end{equation}

    In particular, for any \(k = 1, \ldots, K\), one has
    \[
    f(x^{u+1}, y^{u+1}, a^{u+1}) - f(x^{u}, y^{u}, a^{u}) 
    \geq
    \frac{1}{3}\lvert d^u_k \rvert \min\left(
    \frac{1}{8\eta}
    ,
    \frac{\lvert  d^u_k \rvert}{8\eta c_d + 4\eta(K+L)c_d^2}
    \right).
    \]
    Suppose that there are \(N_I\) iterations for which \(\frac{1}{8\eta} \leq \frac{\lvert  d^u_k \rvert}{8\eta c_d + 4\eta(K+L)c_d^2}\) for some \(k = 1, \ldots, K\). For such \(u\), one has \(\lvert  d^u_k \rvert = \Omega(1)\) and therefore one has \(f(x^{u+1}, y^{u+1}, a^{u+1}) - f(x^{u}, y^{u}, a^{u}) = \Omega(1)\). Thus \(N_I = O(c_g)\), and therefore for \(k = 1, \ldots, K\), one has \(\frac{1}{8\eta} \geq \frac{\lvert  d^u_k \rvert}{8\eta c_d + 4\eta(K+L)c_d^2}\) except for \(O(c_g)\) iterations. Let \(C = \frac{1}{8\eta c_d + 4\eta(K+L)c_d^2}\), and then the following condition holds except for \(N_{I}\) iterations:
    \begin{equation}\label{eqn: generic condition 1 for greedy update}
        f(x^{u+1}, y^{u+1}, a^{u+1}) - f(x^{u}, y^{u}, a^{u}) 
    \geq
    \frac{1}{3}\left(
    \KL{P_u\mathbf{1}}{r} + \KL{P_u^{\top}\mathbf{1}}{c} + C\sum_{m=1}^{K+L}(d_m^u)^2
    \right).
    \end{equation}
    Let \(N_{II}\) be the number of iterations for which the following condition is satisfied 
    \begin{equation}\label{eqn: generic condition 2 for greedy update}
        \lVert \nabla f\rVert_{1} = \lVert P_u\mathbf{1} - r \rVert_{1} + \lVert P_u^{\top}\mathbf{1} - c \rVert_{1} + \sum_{m = 1}^{K+L}\lvert d_m \rvert > \epsilon.
    \end{equation}
    Then, if \(u\) is among the said \(N_{II}\) iterations, one has
    \begin{align*}
        \epsilon^2 < &\left(\lVert P_u\mathbf{1} - r \rVert_{1} + \lVert P_u^{\top}\mathbf{1} - c \rVert_{1} +  \sum_{m = 1}^{K+L}\lvert d^u_m \rvert\right)^2\\
        \leq &(K+L+2)^2\left(\lVert P_u\mathbf{1} - r \rVert_{1}^2 + \lVert P_u^{\top}\mathbf{1} - c \rVert_{1}^2 + \sum_{m=1}^{K+L}(d_m^u)^2\right)\\
        \leq &(K+L+2)^2\left(2 \KL{P_u\mathbf{1}}{r} + 2\KL{P_u^{\top}\mathbf{1}}{c} + \sum_{m=1}^{K+L}(d_m^u)^2\right),
    \end{align*}
    where the second inequality is by Cauchy-Schwartz, the third inequality is \eqref{eqn: def of constraint distance}, and the third inequality is by Pinsker's inequality. Thus, for at least \(N_{II} - N_{I}\) iterations, both \eqref{eqn: generic condition 1 for greedy update} and \eqref{eqn: generic condition 2 for greedy update} are satisfied, under which there exists a constant \(C'\) for which one has
    \begin{align*}
        \epsilon^2 
        \leq &(K+L+2)^2\left(2 \KL{P_u\mathbf{1}}{r} + 2\KL{P_u^{\top}\mathbf{1}}{c} + \sum_{m=1}^{K+L}(d_m^u)^2\right)\\
        \leq & C'\left(f(x^{u+1}, y^{u+1}, a^{u+1}) - f(x^{u}, y^{u}, a^{u}) \right)
    \end{align*}
    As \(\sum_{u = 0}^{\infty}f(x^{u+1}, y^{u+1}, a^{u+1}) - f(x^{u}, y^{u}, a^{u}) \leq c_g\), one must have \(N_{II} - N_{I} = O\left((\epsilon)^{-2}c_g\right)\). One then has \(N_{II} = N_{I} + O\left((\epsilon)^{-2}c_g\right) = O\left((\epsilon)^{-2}c_g\right)\), as desired, and thus the first \(u\) for which \eqref{eqn: greedy main result} holds occurs after \(O\left((\epsilon)^{-2}c_g\right)\) iterations. 
    
    Lastly, \eqref{eqn: greedy main result 2} holds as a result of \eqref{eqn: greedy main result} because one has \(\lvert d^{u}_{K+l} \rvert = P^u \cdot D_{K+l}\) and \(\lvert d^{u}_{k} \rvert \geq \lvert \min(0, P^u \cdot D_{k})\rvert \).
    
\end{proof}

As a result of Theorem \ref{thm: OT linear convergence}, one has the following bound if one applies a rounding operation to the output of \(P\) in Theorem \ref{thm: OT linear convergence}:
\begin{proposition}\label{prop: constraint violation bound}
    Let \(P\) be the same as in Theorem \ref{thm: OT linear convergence}, let \(\mathcal{U}_{r, c}\) denote the set of transport matrix from \(r\) to \(c\), and let \(\mathrm{Round}(P, \mathcal{U}_{r, c})\) be the result of applying the rounding algorithm in \citet{altschuler2017near} to obtain projection of \(P\) into \(\mathcal{U}_{r, c}\). Define constraint violation of \(P\) by \[\begin{aligned}
        \mathrm{Violation}(P) = 
    &\sum_{k=1}^{K}\lvert \min\left(\mathrm{Round}(P, \mathcal{U}_{r, c})\cdot D_k, 0\right)  \rvert\\ + &\sum_{l=1}^{L}\lvert \mathrm{Round}(P, \mathcal{U}_{r, c})\cdot D_{l+K}  \rvert.
    \end{aligned}\]
    One has 
    \[
    \mathrm{Violation}(P) \leq \epsilon\left(1 + 2(K+L)c_d\right)
    \]
    where \(c_{d} := \max_{m \in [K+L]} \lVert D_m \rVert_{\infty}\).
\end{proposition}
\begin{proof}
    This proof is a simple consequence of existing results. By Lemma 7 in \citet{altschuler2017near}, one has
    \[
    \lVert P - \mathrm{Round}(P, \mathcal{U}_{r, c}) \rVert_{1} \leq 2 \epsilon.
    \]
    Thus one has
    \[\begin{aligned}
        \mathrm{Violation}(P) \leq 
    &\sum_{k=1}^{K}\lvert \min\left(P\cdot D_k, 0\right)  \rvert+ \sum_{l=1}^{L}\lvert P\cdot D_{l+K}  \rvert + \sum_{m = 1}^{K+L} \lvert \left(P -\mathrm{Round}(P, \mathcal{U}_{r, c})\right) \cdot D_{m}  \rvert\\
    \leq &\epsilon + 2(K+L)c_d\epsilon,
    \end{aligned}\]
    where the last inequality uses Holder's inequality and \eqref{eqn: greedy main result 2}.
\end{proof}

\section{Equivalence of primal and primal-dual form}\label{sec: primal versus primal dual}
We now show that the primal form in \eqref{eqn: entropic constrained OT general form} can be obtained from the primal-dual form by eliminating the dual variables. 
\begin{proposition}
Define
\begin{align*}
L(P, s, x,y,a) = &\frac{1}{\eta} P \cdot \log P + C\cdot P - x\cdot(P\1-r) - y\cdot(P^{\top}\1-c)\\ + &\frac{1}{\eta}\sum_{k = 1}^{K}s_{k} \log{s_{k}} + \sum_{k=1}^{K}a_{k}s_{k} -\sum_{m = 1}^{K+L}a_{m}(D_{m} \cdot P),
\end{align*}
and then, for \(\mathcal{S}\) as in \eqref{eqn: constrained OT general form}, the following equation holds:
\begin{equation}
\max_{x,y,a}\min_{P,s}  L(P, s, x,y,a,b) = \min_{P: P\1=r, P^{\top}\1=c, P \in \mathcal{S}} \frac{1}{\eta} P\cdot \log P + \sum_{k =1}^{K}\frac{1}{\eta}(D_{k} \cdot P) \cdot \log (D_{k} \cdot P) + C\cdot P
\end{equation}
Moreover, for the Lyapunov potential function \(f\) in \eqref{eqn: dual form}, one has
\begin{equation}
f(x,y,a) =  \min_{P,s} L(P, s, x,y,a).
\end{equation}
\end{proposition}
\begin{proof}
As \(L\) is concave in \(x, y\) and convex in \(P\), one can invoke the minimax theorem to interchange the operations of maximization and minimization. Therefore:
\begin{align*}
    \min_{P,s} \max_{x,y,a} L(P, s, x,y,a) = &\min_{P,s: P\1=r, P^{\top}\1=c, P \in \mathcal{S}, s_{k} =D_{k} \cdot P\, \forall k \in [K]}\frac{1}{\eta}P\cdot \log P + \frac{1}{\eta} s \cdot \log s + C\cdot P\\
    =&\min_{P: P\1=r, P^{\top}\1=c, P \in \mathcal{S}} \frac{1}{\eta} P\cdot \log P + \sum_{k =1}^{K}\frac{1}{\eta}(D_{k} \cdot P) \cdot \log (D_{k} \cdot P) + C\cdot P.
\end{align*}

In terms of entries, one writes \(L(P, s, x,y,a)\) as follows:
\begin{equation}\label{eqn: primal-dual form}
\begin{aligned}
    \max_{x_i,y_j} \min_{p_{ij}} L(P, s, x, y, a) =  &\frac{1}{\eta} \sum_{ij}p_{ij}\log p_{ij} + \sum_{ij} C_{ij} p_{ij} - \sum_i x_i(\sum_j p_{ij}- r_i) - \sum_j y_j (\sum_i p_{ij}- c_j)\\
    + &\frac{1}{\eta}\sum_{k = 1}^{K}s_{k} \log{s_{k}} + \sum_{k=1}^{K}a_{k}s_{k} -\sum_{m = 1}^{K+L}\sum_{ij}a_{m}(D_{m})_{ij} p_{ij}
\end{aligned}
\end{equation}

We then solve the inner min problem explicitly by taking the derivative of $p_{ij}, s_{k}$ to zero, from which one obtains
\[
p_{ij} = \exp(\eta(-C_{ij} + \sum_{m = 1}^{K+L}a_m(D_m)_{ij} + x_i + y_j ) -1).
\]
and
\[
s_{k} = \exp(-\eta a_k - 1).
\]
Plugging in the formula for \(p_{ij}\) and \(s_k\), one has
\begin{align*}
    &\min_{P,s}L(P, s, x,y, a) \\
    &= -\frac{1}{\eta}\sum_{ij}\exp(\eta(-C_{ij}+ \sum_{m = 1}^{K+L}a_m(D_m)_{ij} + x_i + y_j ) -1) + \sum_i r_ix_i + \sum_j c_j y_j - \frac{1}{\eta}\sum_{k = 1}^{K} {\exp(-\eta a_{k} - 1)},
\end{align*}
which is equal to \(f(x, y, a)\).
\end{proof}

\section{Relation of Algorithm \ref{alg:1} to Bregman projection}\label{appendix: bregman}

In this section, we show that the proposed \(a\) update step is equivalent to a Bregman projection for equality constraints. The Bregman projection step introduced in \citet{benamou2015iterative} introduces an iterative projection-based approach. As in the setting of Section \ref{sec: alg}, let \((x, y, a)\) be the current dual variable and let \(P\) be the intermediate matrix corresponding to \((x, y, a)\).

Suppose that \(K = 0\) and thus there are no inequality constraints. As in main text, we define \(\mathcal{E}\) by the following space \[\mathcal{E} := \bigcap_{l = 1, \ldots, L} \{M \mid D_{l} \cdot M = 0\}.\]
Moreover, assume \(\sum_{i} r_i = \sum_{j}c_j = 1\) and define \(\Delta_{n\times n}\) as the \(n^2\)-dimensional simplex.
Let \(a', t' = \argmax_{\tilde{a}, \tilde{t}}f(x + \tilde{t}\mathbf{1},y,\tilde{a})\), and let \(P'\) be the intermediate matrix corresponding to \((x + t'\mathbf{1}, y, a')\). Then, the claimed equivalence can be seen by proving the following equation:
\begin{equation}\label{eqn: Bregman WTS}
    P' = \argmin_{M \in \Delta_{n\times n} \cap \mathcal{E}}\KL{M}{P}
\end{equation}
where for two entry-wise non-negative matrices \(M, N\), the term \(\KL{M}{N}\) is defined by \(\KL{M}{N} = \sum_{ij}m_{ij}\left(\log(\frac{m_{ij}}{n_{ij}}) - 1\right)\). Suppose \(M = [m_{ij}]_{i,j = 1}^{n}\) achieves the optimality condition set in \eqref{eqn: Bregman WTS}. As it is a constrained optimization problem, the necessary condition for optimality is that there exists \(\mu\) and \(\lambda_{l}\) for \(l = 1, \ldots, L\), so that the following holds:
\[
\forall {i,j \in [n]}, \quad \partial_{m_{ij}}\KL{M}{P} = \eta\left(\sum_{l = 1}^{L}\lambda_{l}(D_{l})_{ij} + \mu\right).
\]
Utilizing the definition of \(\KL{\cdot}{\cdot}\), one rewrites the above equation as below:
\[
\log(m_{ij}) - \log(p_{ij}) = \eta\left(\sum_{l = 1}^{L}\lambda_{l}(D_{l})_{ij} + \mu\right).
\]
Thus, there exists \(\mu\) and \(\lambda_{l}\) for \(l = 1, \ldots, L\), for which one has \(M = \exp(\sum_{l}\lambda_{l}D_{m}+ \mu \mathbf{1}\mathbf{1}^{\top}) \odot P\), where \(\odot\) is the Hadamard product. Thus, one has
\begin{equation}\label{eqn: stationarity bregman}
    M_{ij} = \exp{\left(\eta(-C_{ij} + \sum_{l = 1}^{L}(a_{m} + \lambda_{l})\left(D_{l}\right)_{ij}+ x_i + y_{j} + \mu) - 1 \right)}.
\end{equation}

Furthermore, \(M\) satisfies the following normalization condition
\begin{equation}\label{eqn: feasibility 1}
\sum_{ij}M_{ij} = 1.
\end{equation}
Moreover, for \(l \in [L]\), the following holds:
\begin{equation}\label{eqn: feasibility 2}
M \cdot D_{l} = 0.
\end{equation}
As seen in the main text, \(P'\) satisfies \eqref{eqn: feasibility 1}, \eqref{eqn: feasibility 2}. Moreover, by setting \(\lambda_{l} = a'_m - a_m\) and \(\mu = t'\), one can show that \(P'\) also satisfies \eqref{eqn: stationarity bregman}. As \(f\) is concave and the equality constraints are affine, this shows that \(P'\) satisfies the optimality condition in that of the right-hand side of \eqref{eqn: Bregman WTS}. 

\section{Practical implementation of accelerated Sinkhorn-type algorithm under constraint}\label{appendix: SNS et al}
In this section, we detail the procedure to combine Algorithm \ref{alg:1} with entropy regularization scheduling and Sinkhorn-Newton-Sparse (SNS). As the detail of entropic regularization scheduling is presented in Section \ref{sec: alg}, we shall give implementation details of SNS in the constrained case. 
Similar to the construction in \citet{tang2024accelerating}, for a matrix \(M \in \R^{n\times n}_{\geq 0}\), we use \(\mathrm{Sparisfy}(M, \rho)\) to denote entry-wise truncation with a threshold \(\rho\). For \(\Tilde{M} := \mathrm{Sparisfy}(M, \rho)\), one has
\[
\Tilde{M}_{ij} = \begin{cases}
    M_{ij} & \text{if } M_{ij} \geq \rho,\\
    0 & \text{otherwise}.
\end{cases}
\]

Let \(P\) be the intermediate transport matrix formed by the current dual variable \((x, y, a)\). One can write down the Hessian term \(\nabla^2 f\) as follows:
\begin{equation}\label{eqn: Hessian formula}
\nabla^2 f(x, y, a) = -\eta
\begin{bmatrix}
\mathrm{diag}(P\1) & P & \nabla_{x}\nabla_{a}f \\
P^{\top} & \mathrm{diag}(P^{\top}\1) & \nabla_{y}\nabla_{a}f\\
\nabla_{a}\nabla_{x}f & \nabla_{a}\nabla_{y}f & \nabla_{a}^2f
\end{bmatrix}.
\end{equation}
As we assume \(K + L = O(1)\) in this work, the terms such as \(\nabla_{xy} \nabla_{a} f, \nabla_{a}^2 f\) can be kept without posing significant challenges to the Newton step, which allows us to use the approximation below:

\begin{equation}\label{eqn: sparsified Hessian formula}
\nabla^2 f(x, y, a) \approx H = -\eta\begin{bmatrix}
\mathrm{diag}(P\1) & \mathrm{Sparisfy}(P, \rho) & \nabla_{x}\nabla_{a}f \\
\mathrm{Sparisfy}(P^{\top}, \rho) & \mathrm{diag}(P^{\top}\1) & \nabla_{y}\nabla_{a}f\\
\nabla_{a}\nabla_{x}f & \nabla_{a}\nabla_{y}f & \nabla_{a}^2f
\end{bmatrix}.
\end{equation}

The value of \(\rho\) is a tunable parameter, and one sets \(\rho\) so that \(\mathrm{Sparisfy}(P, \rho)\) contains only \(O(n)\) nonzero entries. As a result, applying \(H^{-1}\) to a vector can be done with a \(O((n+K+L)^2) = O(n^2)\) complexity through the conjugate-gradient algorithm \citep{golub2013matrix}. Additionally, to ensure stability in the degenerate direction \((\delta x, \delta y, \delta a) = v:= (\mathbf{1}_n, -\mathbf{1}_n, \mathbf{0}_{K+L})\), we in practice use the following modified version of Lyapunov function:
\[
\tilde{f}(x, y, a) = f(x, y, a) - \frac{1}{2}(\sum_{i}x_{i} - \sum_{j}y_{j})^2.
\]
Same as \citet{tang2024accelerating}, one can see that the maximizer of \(\Tilde{f}\) is also a maximizer of \(f\). The final Sinkhorn-Newton-Sparse algorithm used is in Algorithm \ref{alg:SNS augmented}, where we include a Sinkhorn stage which uses Algorithm \ref{alg:1} as initialization, and a subsequent Newton stage which uses sparse Newton iteration to accelerate convergence.

\begin{algorithm}
\caption{Sinkhorn-Newton-Sparse (SNS) for OT under linear constraint}\label{alg:SNS augmented}
\begin{algorithmic}[1]
\Require $\Tilde{f}, x_{\mathrm{init}} \in \mathbb{R}^{n}, y_{\mathrm{init}} \in \mathbb{R}^{n}, a_{\mathrm{init}} \in \mathbb{R}^{K+L}, N_1, N_2, \rho, i = 0$
\State \texttt{\# Sinkhorn stage}

\State $v \gets \begin{bmatrix}
\1_{n}\\
-\1_{n}\\
\mathbf{0}_{K+L}
\end{bmatrix}$ \Comment{Initialize degenerate direction}

\State $(x,y, a) \gets (x_{\mathrm{init}}, y_{\mathrm{init}}, a_{\mathrm{init}}) $ \Comment{Initialize dual variable}
\While{$i < N$}
\State $i \gets i + 1$

\State \texttt{\# Row\&Column scaling step}
    \State \(P \gets \exp{\left(\eta(-C + \sum_m a_mD_m + x\1^{\top} + \1y^{\top}) - 1\right)}\)
    \State \( x \gets x + \left(\log(r) - \log(P\mathbf{1} ) \right)/\eta\) 
    \State \(P \gets \exp{\left(\eta(-C + \sum_m a_mD_m + x\1^{\top} + \1y^{\top}) - 1\right)}\)
    \State \( y \gets y + \left(\log(c) - \log(P^{\top}\mathbf{1} ) \right)/\eta\)
\State \texttt{\# Constraint dual update step}
\State \(a, t \gets \argmax_{\tilde{a}, \tilde{t}}f(x + \tilde{t} \mathbf{1},y,\tilde{a})\)
\State \(x \gets x + t\mathbf{1}\)
\EndWhile
\State \texttt{\# Newton stage}
\State $z \gets \mathrm{Proj}_{v^{\perp}}((x,y, a))$
\Comment{Project into non-degenerate direction of \(f\)}
\While{$i < N_1 + N_2$} 
\State \(P \gets \exp{\left(\eta(-C + \sum_m a_mD_m + x\1^{\top} + \1y^{\top}) - 1\right)}\)
\State \(H \gets -\eta\begin{bmatrix}
\mathrm{diag}(P\1) & \mathrm{Sparisfy}(P, \rho) & \nabla_{x}\nabla_{a}f \\
\mathrm{Sparisfy}(P^{\top}, \rho) & \mathrm{diag}(P^{\top}\1) & \nabla_{y}\nabla_{a}f\\
\nabla_{a}\nabla_{x}f & \nabla_{a}\nabla_{y}f & \nabla_{a}^2f
\end{bmatrix}\) \Comment{Sparse approximation of \(\nabla^{2}f\) with threshold \(\rho\).}
    \State \(H \gets H - vv^{\top}\) \Comment{Add regularization term corresponding to \(\Tilde{f}\).}
    \State $\Delta z \gets \text{Conjugate\_Gradient}(H, -\nabla \Tilde{f}(z))$ \Comment{Solve sparse linear system}
    \State $\alpha \gets \text{Line\_search}(\Tilde{f}, z, \Delta z) $ \Comment{Line search for step size}
    \State $z \gets z + \alpha  \Delta z$
    \State $i \gets i + 1$
\EndWhile
\State Output dual variables $(x,y,a) \gets z$.
\end{algorithmic}
\end{algorithm}

The combination of entropy regularization scheduling and SNS is described in Algorithm \ref{alg:SNS+annealing}. We remark that the iteration count \(N_{1, i}, N_{2, i}\) within Algorithm \ref{alg:SNS+annealing} are typically set to be much smaller than \(N_1, N_2\) in Algorithm \ref{alg:SNS augmented}, which is possible because the optimization task at an entropy regularization parameter \(\eta_{i}\) is initialized by the dual variables obtained for regularization parameter \(\eta_{i-1} = \eta_{i}/2\). 

As a special case, Algorithm \ref{alg:SNS+annealing} also provides substantial acceleration to the unconstrained entropic OT problem \citep{cuturi2013sinkhorn}.

\begin{algorithm}
\caption{Sinkhorn-Newton-Sparse with entropy regularization scheduling for OT under linear constraint}\label{alg:SNS+annealing}
\begin{algorithmic}[1]
\Require $x_{\mathrm{init}} \in \mathbb{R}^{n}, y_{\mathrm{init}} \in \mathbb{R}^{n}, a_{\mathrm{init}} \in \mathbb{R}^{K+L}, \rho$
\Require $\eta_{\mathrm{target}}, N_{\eta}= \ceil{\log_{2}(\eta_{\mathrm{target}})}, (N_{1, i})_{i=1}^{N_{\eta}}, (N_{2, i})_{i=1}^{N_{\eta}}, i = 1$

\State $(x,y, a) \gets (x_{\mathrm{init}}, y_{\mathrm{init}}, a_{\mathrm{init}}) $ \Comment{Initialize dual variable}
\State $\eta = 1$ \Comment{Initialize entropy regularization}
\While{$i \leq N_{\eta}$} 
\State Run Algorithm \ref{alg:SNS augmented} with entropy regularization set to \(\eta\) and initialized dual variables set to \((x, y, a)\), and \(N_1, N_2\) set to \(N_{1, i}, N_{2, i}\).
\State Save the output of previous step to $(x, y, a)$.
\State $i \gets i + 1$
\State $\eta = \min(2\eta, \eta_{\mathrm{target}})$ \Comment{Double entropy regularization term}
\EndWhile
\State Run the Newton stage of Algorithm \ref{alg:SNS augmented} at \(\eta = \eta_{\mathrm{target}}\) until the solution reaches convergence.
\State Output dual variables $(x,y,a)$.
\end{algorithmic}
\end{algorithm}

\section{Further details on numerical experiments }\label{appendix: empirical exponential convergence}
Despite the theoretical analysis in this work, it is an open problem as to whether the intermediate matrix \(P\) formed by the dual variable \((x,y,a)\) converges to \(P_{\eta}^{\star}\), which is the entropic optimal solution to \eqref{eqn: entropic constrained OT general form}. While the theoretical result is left for future work, we show that the numerical performance of Algorithm \ref{alg:1} lends strong support for rapid practical convergence. In Figure \ref{fig:tv_distance_linear_assignment} and Figure \ref{fig:tv_distance_dcg}, we show that both Algorithm \ref{alg:1} and Algorithm \ref{alg:SNS augmented} can converge to \(P_{\eta}^{\star}\) in the total variation (TV) distance. Furthermore, Algorithm \ref{alg:SNS augmented} provides machine-epsilon accurate solutions rapidly during the Newton stage. 

We further remark that the ranking example seems to suggest an exponential convergence result exists in practice for Algorithm \ref{alg:1} in this setting. After conducting repeated experiments in the ranking setting under different problem parameters, the numerical result obtained shows that exponential convergence is not peculiar to the experiment but is instead a persistent feature in the ranking case. The analysis in this special case might be of independent research interest for the field of information retrieval.
\begin{figure}[ht]
    \centering
    \includegraphics[width = 0.9\textwidth]{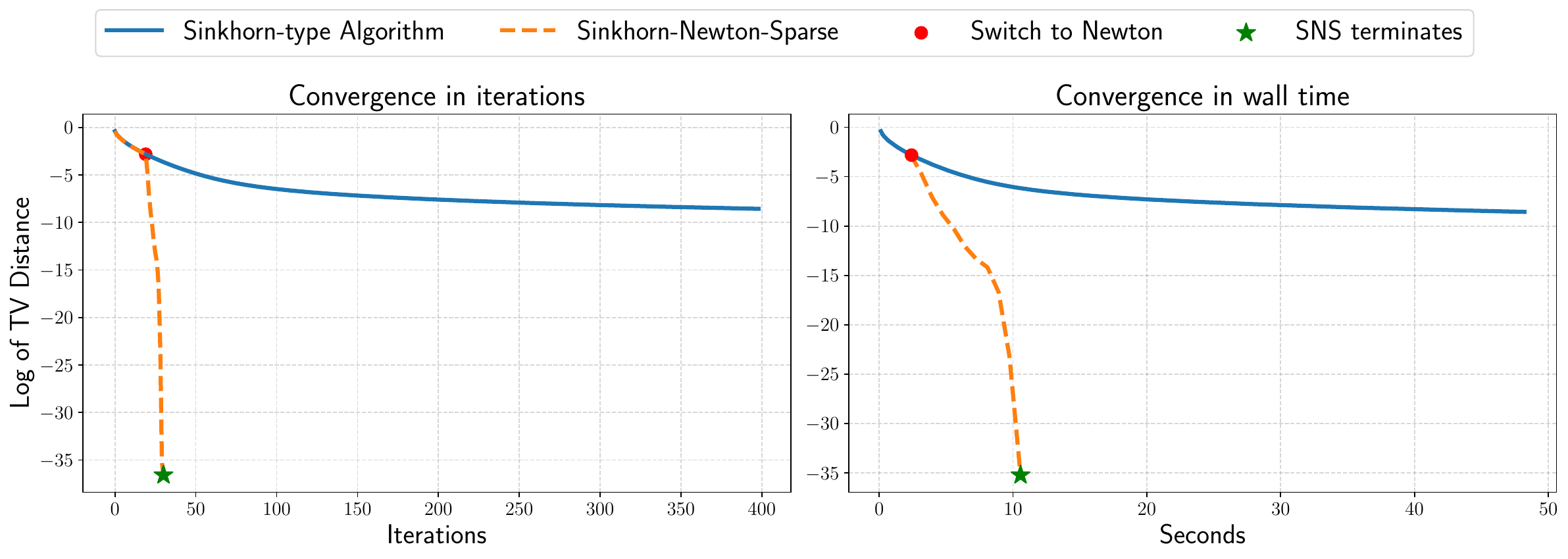}
    \caption{Random assignment problem. Plot of the proposed Sinkhorn-type algorithm in terms of TV distance to \(P_{\eta}^{\star}\). }
    \label{fig:tv_distance_linear_assignment}
    \centering
    \includegraphics[width = 0.9\textwidth]{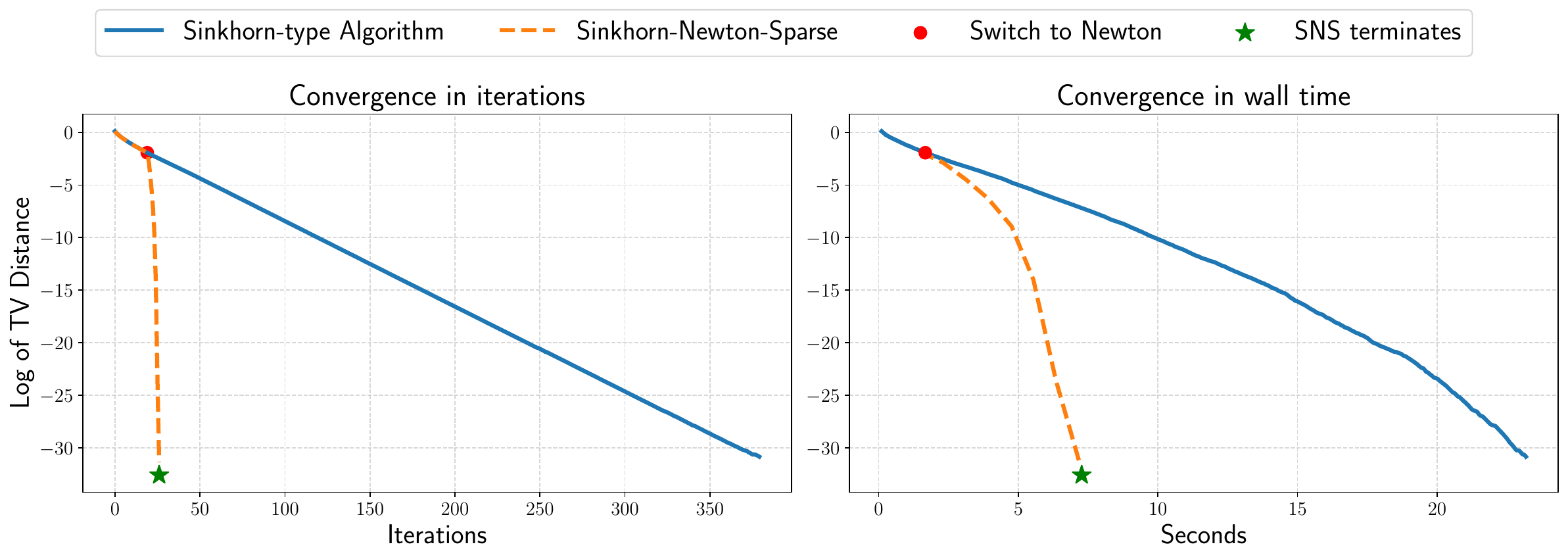}
    \caption{Ranking under constraints. Plot of the proposed Sinkhorn-type algorithm in terms of TV distance to \(P_{\eta}^{\star}\).}
    \label{fig:tv_distance_dcg}
\end{figure}

\end{document}